\documentclass[a4paper, twoside, 11pt, english]{article}

\usepackage[T1]{fontenc}
\usepackage[utf8]{inputenc}

\usepackage{etoolbox}
\usepackage{amsmath,amsthm,amssymb,stmaryrd}
\usepackage{mathtools}
\usepackage[ttscale=.875]{libertine}
\usepackage[libertine]{newtxmath}

\usepackage[numbers,sort&compress]{natbib}
\usepackage{comment}
\usepackage{ifthen}
\usepackage{algorithm2e}

\usepackage{fancyhdr, titlesec, url, enumerate, microtype,setspace}
\usepackage{multicol}
\usepackage[all,knot,poly]{xy}
\usepackage{tikz}
\usetikzlibrary{decorations.pathreplacing}

\usepackage{tikz-qtree}
\usepackage{youngtab}
\usepackage{ytableau}

\usepackage[english]{babel}

\usepackage{hyperref}

\CompileMatrices

\usetikzlibrary{calc}

\newcount\hh
\newcount\mm
\mm=\time
\hh=\time
\divide\hh by 60
\divide\mm by 60
\multiply\mm by 60
\mm=-\mm
\advance\mm by \time
\def\hhmm{\number\hh:\ifnum\mm<10{}0\fi\number\mm}

\newcommand{\periodafter}[1]{\ifstrempty{#1}{}{#1.}}
\titleformat{\section}[block]{\scshape\filcenter\LARGE}{\thesection.}{.5em}{}
\titleformat{\subsection}[block]{\bfseries\filcenter\large}{\thesubsection.}{.5em}{\medskip}
\titleformat{\subsubsection}[runin]{\bfseries}{\thesubsubsection.}{.5em}{\periodafter}
\titlespacing{\subsubsection}{0pt}{\topsep}{.5em}

\newtheoremstyle{ntheorem}%
	{\topsep}{\topsep}{\itshape}{0pt}{\bfseries}{.}{.5em}%
	{\thmnumber{#2.\hspace{.5em}}\thmname{#1}\thmnote{ (#3)}}
	
\newtheoremstyle{ndefinition}%
	{\topsep}{\topsep}{\normalfont}{0pt}{\bfseries}{.}{.5em}%
	{\thmnumber{#2.\hspace{.5em}}\thmname{#1}\thmnote{ (#3)}}
	
\newtheoremstyle{nremark}%
	{\topsep}{\topsep}{\normalfont}{0pt}{\itshape}{.}{.5em}%
	{\thmnumber{}\thmname{#1}\thmnote{ (#3)}}

\theoremstyle{ndefinition}

	\makeatletter
\def\@equationname{equation}

\makeatother

\fancypagestyle{plain}{
\fancyhf{}\fancyfoot[LC,RC]{}
\fancyfoot[LE,RO]{$\thepage$}

}

\UseTips
\SelectTips{eu}{11}

\newdir{ >}{{}*!/-10pt/@{>}}
\newdir{ -}{{}*!/-10pt/@{}}
\newdir{> }{{}*!/+10pt/@{>}}

\makeatletter

\xyletcsnamecsname@{dir4{}}{dir{}}
\xydefcsname@{dir4{-}}{\line@ \quadruple@\xydashh@}
\xydefcsname@{dir4{.}}{\point@ \quadruple@\xydashh@}
\xydefcsname@{dir4{~}}{\squiggle@ \quadruple@\xybsqlh@}
\xydefcsname@{dir4{>}}{\Tttip@}
\xydefcsname@{dir4{<}}{\reverseDirection@\Tttip@}

\xydef@\quadruple@#1{%
	\edef\Drop@@{%
		\dimen@=#1\relax
		\dimen@=.5\dimen@
		\A@=-\sinDirection\dimen@
		\B@=\cosDirection\dimen@
		\setboxz@h{%
			\setbox2=\hbox{\kern3\A@\raise3\B@\copy\z@}%
			\dp2=\z@ \ht2=\z@ \wd2=\z@ \box2
			\setbox2=\hbox{\kern\A@\raise\B@\copy\z@}%
			\dp2=\z@ \ht2=\z@ \wd2=\z@ \box2
			\setbox2=\hbox{\kern-\A@\raise-\B@\copy\z@}%
			\dp2=\z@ \ht2=\z@ \wd2=\z@ \box2
			\setbox2=\hbox{\kern-3\A@\raise-3\B@ \noexpand\boxz@}%
			\dp2=\z@ \ht2=\z@ \wd2=\z@ \box2
		}%
		\ht\z@=\z@ \dp\z@=\z@ \wd\z@=\z@ \noexpand\styledboxz@
	}%
}

\xydef@\Tttip@{\kern2pt \vrule height2pt depth2pt width\z@
	\Tttip@@ \kern2pt \egroup
	\U@c=0pt \D@c=0pt \L@c=0pt \R@c=0pt \Edge@c={\circleEdge}%
	\def\Leftness@{.5}\def\Upness@{.5}%
	\def\Drop@@{\styledboxz@}\def\Connect@@{\straight@{\dottedSpread@\jot}}}
	
\xydef@\Tttip@@{%
	\dimen@=.25\dimen@
 	\B@=\cosDirection\dimen@
	\setboxz@h\bgroup\reverseDirection@\line@ \wdz@=\z@ \ht\z@=\z@ \dp\z@=\z@
	{\vDirection@(1,-1)\xydashl@ \xyatipfont\char\DirectionChar}%
	{\vDirection@(1,+1)\xydashl@ \xybtipfont\char\DirectionChar}%
}

\xydef@\ar@form{
	\ifx \space@\next \expandafter\DN@\space{\xyFN@\ar@form}%
	\else\ifx ^\next \DN@ ^{\xyFN@\ar@style}\edef\arvariant@@{\string^}%
	\else\ifx _\next \DN@ _{\xyFN@\ar@style}\edef\arvariant@@{\string_}%
	\else\ifx 0\next \DN@ 0{\xyFN@\ar@style}\def\arvariant@@{0}%
	\else\ifx 1\next \DN@ 1{\xyFN@\ar@style}\def\arvariant@@{1}%
	\else\ifx 2\next \DN@ 2{\xyFN@\ar@style}\def\arvariant@@{2}%
	\else\ifx 3\next \DN@ 3{\xyFN@\ar@style}\def\arvariant@@{3}%
	\else\ifx 4\next \DN@ 4{\xyFN@\ar@style}\def\arvariant@@{4}%
	\else\ifx \bgroup\next \let\next@=\ar@style
	\else\ifx [\next \DN@[##1]{\ar@modifiers{[##1]}}
	\else\ifx *\next \DN@ *{\ar@modifiers}%
	\else\addLT@\ifx\next \let\next@=\ar@slide
	\else\ifx /\next \let\next@=\ar@curveslash
	\else\ifx (\next \let\next@=\ar@curveinout 
	\else\addRQ@\ifx\next \addRQ@\DN@{\ar@curve@}%
	\else\addLQ@\ifx\next \addLQ@\DN@{\xyFN@\ar@curve}%
	\else\addDASH@\ifx\next \addDASH@\DN@{\defarstem@-\xyFN@\ar@}%
	\else\addEQ@\ifx\next \addEQ@\DN@{\def\arvariant@@{2}\defarstem@-\xyFN@\ar@}%
	\else\addDOT@\ifx\next \addDOT@\DN@{\defarstem@.\xyFN@\ar@}%
	\else\ifx :\next \DN@:{\def\arvariant@@{2}\defarstem@.\xyFN@\ar@}%
	\else\ifx ~\next \DN@~{\defarstem@~\xyFN@\ar@}%
	\else\ifx !\next \DN@!{\dasharstem@\xyFN@\ar@}%
	\else\ifx ?\next \DN@?{\ar@upsidedown\xyFN@\ar@}%
	\else \let\next@=\ar@error
	\fi\fi\fi\fi\fi\fi\fi\fi\fi\fi\fi\fi\fi\fi\fi\fi\fi\fi\fi\fi\fi\fi\fi \next@}

\makeatother

\newcommand{\fl}{\rightarrow}

\newcommand{\dfl}{\Rightarrow}

\newcommand{\tfl}{\Rrightarrow}
\newcommand{\qfl}{\xymatrix@1@C=10pt{\ar@4 [r] &}}

\newcommand{\tck}[1]{#1^{\top}}

\renewcommand{\phi}{\varphi}
\renewcommand{\epsilon}{\varepsilon}

\newcommand{\Zb}{\mathbb{Z}}

\newcommand{\C}{\mathbf{C}}

\newcommand{\M}{\mathbf{M}}

\def\catego#1{{\bf{\sf #1}}}

\newcommand{\Cat}{\catego{Cat}}
\newcommand{\Act}{\catego{Act}}

\newcommand{\ifthen}[2]{\ifthenelse{#1}{#2}{}}

\definecolor{Blue}{rgb}{0.0,1.0,1.0}
\definecolor{cyan}{RGB}{175,238,238} 
\definecolor{GreenL}{rgb}{0.0,1.0,0.0}

\newcommand{\Cdot}{\!\cdot\!}

\renewcommand{\leq}{\leqslant}
\renewcommand{\geq}{\geqslant}

\newcommand{\lecCol}[1]{\llbracket #1\rrbracket}

\DeclareMathOperator{\PreColo}{PreCol}

\DeclareMathOperator{\Chin}{Ch}

\def\Ich{r}
\def\Jch{l}

\DeclareMathOperator{\ChinT}{Ch}

\DeclareMathOperator{\R}{R}

\makeatletter
\def\blfootnote{\xdef\@thefnmark{}\@footnotetext}
\makeatother

\def\Xr{\mathcal{X}}

\def\len{\ell}

\newcommand{\insl}[1]{\rightsquigarrow_{#1}}
\newcommand{\insr}[1]{\;\raisebox{0.1em}{{\rotatebox[origin=c]{180}{$\rightsquigarrow$}}}_{#1}\;}

\def\ordrecolChinese{\preccurlyeq_{\text{Ch}}}

\newcount\cols
{\catcode`,=\active\catcode`|=\active
 \gdef\Young#1{\hbox{$\vcenter
 {\mathcode`,="8000\mathcode`|="8000
  \def,{\global\advance\cols by 1 &}%
  \def|{\cr
        \multispan{\the\cols}\hrulefill\cr
        &\global\cols=2 }%
  \offinterlineskip\everycr{}\tabskip=0pt
  \dimen0=\ht\strutbox \advance\dimen0 by \dp\strutbox
  \halign
   {\vrule height \ht\strutbox depth \dp\strutbox##
    &&\hbox to \dimen0{\hss$##$\hss}\vrule\cr
    \noalign{\hrule}&\global\cols=2 #1\crcr
    \multispan{\the\cols}\hrulefill\cr%
   }
 }$}}
\gdef\Skew(#1:#2){\hbox{$\vcenter
{\mathcode`,="8000\mathcode`|="8000
  \dimen0=\ht\strutbox \advance\dimen0 by \dp\strutbox
  \def\boxbeg{\vbox
    \bgroup\hrule\kern-0.4pt\hbox to\dimen0\bgroup\strut\vrule\hss$}%
  \def\boxend{$\hss\egroup\hrule\egroup}%
  \def,{\boxend\boxbeg}%
  \def|##1:{\boxend\vrule\egroup\nointerlineskip\kern-0.4pt
    \moveright##1\dimen0\hbox\bgroup\boxbeg}%
  \def\\##1\\##2:{\boxend\vrule\egroup\nointerlineskip\kern-0.4pt
    \kern ##1\dimen0\moveright##2\dimen0\hbox\bgroup\boxbeg}%
  \moveright#1\dimen0\hbox\bgroup\boxbeg#2\boxend\vrule\egroup
 }$}}
}

\definecolor{vert}{rgb}{0,0.45,0}

\DeclareMathOperator{\Srs}{\mathcal{R}}

\titleformat{\section}[block]{\scshape\filcenter\Large}{\thesection.}{.5em}{}
\titleformat{\subsection}[runin]{\bfseries}{\thesubsection.}{.5em}{}[.]
\titleformat{\subsubsection}[runin]{\bfseries}{\thesubsubsection.}{.5em}{}[.]
\titlespacing{\subsubsection}{0pt}{10pt}{.5em}

\theoremstyle{ntheorem}
  	\newtheorem{theorem}[subsection]{Theorem}
  	\newtheorem{proposition}[subsection]{Proposition}
	\newtheorem{lemma}[subsection]{Lemma}
  	\newtheorem{corollary}[subsection]{Corollary}

\begin{document}
\thispagestyle{empty}

\begin{center}

\begin{doublespace}
\begin{huge}
{\scshape Chinese syzygies by insertions}

\end{huge}

\bigskip
\hrule height 1.5pt 
\bigskip

\begin{Large}
{\scshape Nohra Hage \qquad Philippe Malbos}
\end{Large}
\end{doublespace}

\vskip+55pt

\begin{small}\begin{minipage}{14cm}
\noindent\textbf{Abstract --}
We construct a finite convergent semi-quadratic presentation for the Chinese monoid by adding column generators and using combinatorial properties of  insertion algorithms on Chinese staircases.
We extend this presentation into a coherent one whose generators are columns, rewriting rules are defined by insertion algorithms, and whose syzygies are defined as relations among insertion algorithms.
 Such a coherent  presentation is used for representations of Chinese monoids, in particular, it is a way to describe actions of Chinese monoids on categories.

\vskip+10pt


\vskip+3pt

\smallskip\noindent\textbf{M.S.C. 2010 -- Primary:} 20M05, 16S15. \textbf{Secondary:} 68Q42, 20M35.
\end{minipage}\end{small}

\end{center}

\tikzset{every tree node/.style={minimum width=1em,draw,circle},
         blank/.style={draw=none},
         edge from parent/.style=
         {draw,edge from parent path={(\tikzparentnode) -- (\tikzchildnode)}},
         level distance=0.8cm}

\vskip+40pt

\section{Introduction}

The structure of Chinese monoids appeared in the classification of monoids with the growth function coinciding with that of plactic monoids,~\cite{DuchampKrob94}. The latter monoids emerged from the works of Schensted~\cite{Schensted61} and Knuth~\cite{Knuth70} on the combinatorial study of Young tableaux and they have found several applications in algebraic combinatorics, representation theory and probabilistic combinatorics,~\cite{Lothaire02,Fulton97,OConnell03}. One of the motivations for studying Chinese monoids is that they are also related to Young tableaux, and therefore they might play a similar role as plactic monoids  in the same application areas. 
Representations of  the Chinese monoid of finite rank are studied in~\cite{KubatOkninski16} by constructing all its irreducible representations. More generally, we are interested in the study of the actions of Chinese monoids on categories. One approach is to make explicit coherent presentations whose relations are described by insertion algorithms and whose syzygies are defined as relations among these relations.
In this article, we compute coherent presentations for the Chinese monoid by rewriting methods using combinatorial properties of Chinese staircases.
\medskip

This work is a part of a broader project that consists of studying,  by a rewriting approach, families of monoids defined from combinatorial objects constructed using insertion algorithms. For instance,  plactic monoids of type A are related to Young tableaux~\cite{LascouxSchutsenberger81},  plactic monoids of classical types to symplectic and orthogonal tableaux, \cite{Lecouvey02, Lecouvey03}, Chinese monoids to Chinese staircases,~\cite{DuchampKrob94,CassaigneEspieKrobNovelliHibert01}, hypoplactic monoids to quasi-ribbon tableaux,~\cite{Novelli00}, left and right patience sorting monoids to left and right patience sorting tableaux,~\cite{THOMASYong11,CainMalheiroSilva19},   and stalactic monoids to stalactic tableaux~\cite{HivertNovelliThibon07,Priez2013}. Moreover, binary search trees, binary search trees
with multiplicities and pairs of twin binary search
trees are used to describe normal forms for sylvester monoids,~\cite{HivertNovelliThibon05},  taiga monoids,~\cite{Priez2013}, and Baxter monoids,~\cite{Giraudo12}. We are interested in the study of the syzygies for the presentations of theses monoids by computing finite coherent convergent  presentations. Such coherent presentations  are constructed for Artin monoids in~\cite{GaussentGuiraudMalbos15} and for plactic monoids of type~A in~\cite{HageMalbos17}.
The study of the syzygies in a monoid produces in higher dimensions free objects that are homotopically equivalent to the original monoid and then allows computation of its \emph{homological invariants}. Indeed, this study provides the first two steps in the computation of  a \emph{polygraphic resolution} of the monoid, that is, a categorical cofibrant replacement of the monoid in a free $(\omega,1)$-category, whose acyclicity is proved by an iterative construction of a normalization reduction strategy,~\cite{GuiraudMalbos12advances,GuiraudMalbos10smf}. Moreover, coherent presentations are also useful to describe the notion of an \emph{action of the monoid on categories},~\cite{GaussentGuiraudMalbos15}. 
\medskip

The \emph{Chinese monoid of rank $n>0$},  denoted by~$\C_n$,  is generated by~$[n]:=\{1<\ldots<n\}$ and submitted to the relations~$zyx=zxy=yzx$  for all~$1\leq x\leq y\leq z \leq n$. These relations generate the \emph{Chinese congruence}, denoted by~$\sim_{\C_n}$,  and interpreted in~\cite{CassaigneEspieKrobNovelliHibert01} using the notion of Chinese staircases.
A \emph{Chinese staircase} is a collection of boxes in right-justified rows,  filled with non-negative integers, whose rows and  columns are indexed with~$[n]$ from top to bottom and from right to left respectively and where the $i$-th row contains $i$ boxes, for~$1 \leq i \leq n$.
We will denote by~$\ChinT_n$ the set of  Chinese staircases over~$[n]$ and by~$\R$ the map on~$\ChinT_n$ that reads a Chinese staircase row by row from right to left and from top to bottom as defined in Subsection~\ref{SS:Chinesestaircases}.  A Schensted-like insertion algorithm, denoted by~$\insr{\Ich}$, is introduced in~\cite{CassaigneEspieKrobNovelliHibert01}, and consists in inserting an element of~$[n]$ into  a Chinese staircase from the right, yielding to a new Chinese staircase. 
From a word $w=x_1x_2\ldots x_k$ on~$[n]$, we associate a Chinese staircase~$\lecCol{w}_{\Ich}$ obtained by insertion of~$w$ in the empty staircase~$\emptyset$ by application of~$\insr{\Ich}$ step by step from left to right:
\[
\lecCol{w}_{\Ich}:=
(\emptyset \insr{\Ich} w)
= 
(\ldots((\emptyset \insr{\Ich} x_1) \insr{\Ich} x_2) \insr{\Ich} \ldots ) \insr{\Ich} x_k.
\]
Similarly, a Chinese staircase denoted by~$\lecCol{w}_{\Jch}$ is computed by inserting the elements of~$w$ from right to left   in the empty staircase~$\emptyset$ by application of the left insertion~$\insl{\Jch}$ introduced in~\cite{CainGrayMalheiroChinese} and that inserts an element of~$[n]$ into a Chinese staircase from the left.
The set of Chinese staircases satisfies the \emph{cross-section property} for the Chinese congruence~$\sim_{\C_n}$, that is,  for all words~ $w,w'$ on $[n]$, $w\sim_{\C_n} w'$ if and only if the insertion algorithm yields the same Chinese staircase: $\lecCol{w}_{\Ich} = \lecCol{w'}_{\Ich}$,~\cite{CassaigneEspieKrobNovelliHibert01}. So the elements of the Chinese monoid can be identified with the Chinese staircases, which therefore also form a monoid.
Moreover, the right and left insertion algorithms allow one to define two internal products on $\ChinT_n$ by setting~$t \star_{\Ich} t' = (t \insr{\Ich} \R(t'))$ and~$t \star_{\Jch} t' = (\R(t') \insl{\Jch} t)$, for all~$t,t'$ in $\ChinT_n$.
Following the cross-section property,  the compositions~$\star_{\Ich}$  and~$\star_{\Jch}$ are associative and the following equality
\[
y \insl{\Jch} (t \insr{\Ich} x )  
\:=\:
(y \insl{\Jch} t) \insr{\Ich} x
\]
holds, for all~$t$ in~$\ChinT_n$ and $x,y$ in~$[n]$. 
In particular,  the following equality~$\lecCol{w}_{\Ich} = \lecCol{w}_{\Jch}$ holds, for any word $w$ on~$[n]$. 
In this way, the products~$\star_{\Ich}$ and~$\star_{\Jch}$ equip the set~$\ChinT_n$ with two monoid structures that are anti-isomorphic.

We construct in Section~\ref{S:SemiQuadraticConvergentPresentationChineseMonoids} a finite semi-quadratic convergent presentation for the monoid~$\C_n$, denoted by~$\Srs(Q_n, \C_n)$,  whose set of generators~$Q_n$ is made of \emph{columns} over $[n]$ of length at most~$2$ and \emph{square generators} and whose rules are 
\[
\gamma_{u,v}:c_u\Cdot c_v\dfl c_{e}\Cdot c_{e'}
\]
for all~$c_u,c_v$ in~$Q_n$ such that~$c_u\Cdot c_v$  does not form a  Chinese staircase and~$c_u\star_{\Ich} c_v$ is equal to the Chinese staircase composed by the columns~$c_e$ and~$c_{e'}$. We show that this rewriting system can be obtained from the Knuth-like presentation of~$\C_n$ by applying
\emph{Tietze transformations}  that consist in adding or removing definable generators and in adding or removing derivable relations on a presentation of a monoid in such a way that they do not change the presented monoid, see~\cite{GaussentGuiraudMalbos15}. Moreover, we show that the confluence of the rewriting system~$\Srs(Q_n, \C_n)$ is a  direct consequence of the associativity of the product~$\star_{\Ich}$. 
We deduce that the  monoid~$\C_n$ has finite derivation type~$\mathrm{FDT}_{\infty}$ and  finite homological type~$\mathrm{FP}_{\infty}$. 
Note that the finite convergent presentations of Chinese monoids  already obtained in~\cite{ChenQiu08,Karpuz10}, by completion of Chinese relations, and in~\cite{CainGrayMalheiroChinese} by adding column generators, are not semi-quadratic, and thus it is difficult to extend them into coherent ones.

We extend in Section~\ref{S:ChineseSyzygies} the rewriting system $\Srs(Q_n,\C_n)$ into a finite coherent convergent presentation of the Chinese monoid $\C_n$ with an explicit description of the Chinese syzygies. We show in Theorem \ref{T:CoherenceQCol3} that~$\Srs(Q_n, \C_n)$  extends into a finite convergent coherent presentation of the  monoid $\C_n$ by adjunction of generating syzygies with the following decagonal form
\[
\xymatrix @C=2.8em @R=0.65em {
&
c_{e}\Cdot  c_{e'}\Cdot c_t
   \ar@2[r] ^-{\gamma_{e,\widehat{e',t}}} 
 \ar@3 []!<95pt,-10pt>;[dd]!<95pt,10pt> ^{\;\mathcal{X}_{u,v,t}} 
&
c_{e}\Cdot c_{b}\Cdot c_{b'}
   \ar@2[r] ^-{\gamma_{\widehat{e,b},b'}} 
&
c_{s}\Cdot c_{s'}\Cdot c_{b'}
   \ar@2[r] ^-{\gamma_{s,\widehat{s',b'}}}
&
{c_s\Cdot c_k\Cdot  c_{k'}}
 \ar@2@/^/[dr] ^-{\gamma_{\widehat{s,k},k'}}
\\
c_u\Cdot c_v\Cdot c_t
   \ar@2@/^/[ur] ^-{\gamma_{\widehat{u,v},t}}
   \ar@2@/_/[dr] _-{\gamma_{u,\widehat{v,t}}}
&&&&&
c_{\Jch}\Cdot c_{m}\Cdot c_{k'}
\\
&
c_u\Cdot c_{w}\Cdot c_{w'}
   \ar@2[r] _-{\gamma_{\widehat{u,w},w'}}
&
c_{a}\Cdot c_{a'}\Cdot c_{w'}
  \ar@2[r] _-{\gamma_{a,\widehat{a',w'}}}
&
c_{a}\Cdot c_{d}\Cdot c_{d'}
  \ar@2[r] _-{\gamma_{a,\widehat{a',w'}}}
& c_{\Jch}\Cdot c_{l'}\Cdot c_{d'}
  \ar@2@/_/[ur] _-{\gamma_{l,\widehat{l',d'}}}
}	
\]
for all~$c_u,c_v,c_t$ in~$Q_n$ such that~$c_u\Cdot c_v$ and~$c_v\Cdot c_t$ are not  normal forms with respect to~$\Srs(Q_n, \C_n)$, and where the $2$-cells~$\gamma_{-,-}$ denote either a rewriting rule of $\Srs(Q_n, \C_n)$ or an identity.
We show in Subsection~\ref{SS:RelationsAmongInsertionsAlgorithms} how the generating syzygy of the coherent presentation of the Chinese monoid can be interpreted in terms of the right and left insertion algorithms.
Finally, we use in Subsection~\ref{SSS:ActionsChineseMonoids} this coherent presentation in order to describe 
 the actions of  Chinese monoids on categories.

\section{Preliminaries on rewriting}
\label{S:PreliminariesRewriting}

This preliminary section recalls the basic notions of rewriting we use in this article. For a fuller account of the theory, we refer the reader to~\cite{BookOtto93}. We will also recall from \cite{GaussentGuiraudMalbos15,GuiraudMalbos18} the notion of coherent presentation of a monoid that extends the notion of a presentation by syzygies taking into account all the relations amongst the relations.
We will denote by $X^\ast$ the free monoid of \emph{words} written in the alphabet $X$, the product being concatenation of words, and the identity being the empty word, denoted by $\lambda$. 
We will denote by $u=x_1\ldots x_k$ a word in $X^\ast$ of \emph{length} $k$, where~$x_1,\ldots,x_k$ belong to~$X$, and   by~$|u|$ its  length. 

\subsection{String rewriting systems}
A \emph{(string) rewriting system} on $X$ is a subset $R$ of $X^\ast\times X^\ast$. An element~$\beta =(u,v)$ of $R$ is called a \emph{rule} with \emph{source}~$u$ and \emph{target}~$v$, and denoted by $\beta: u\fl v$. We will denote respectively by~$s(\beta)$  and~$t(\beta)$ the source and target of~$\beta$. A \emph{one step reduction} is defined by $wuw'\fl wvw'$ for all words~$w,w'$ in $X^\ast$ and rule $\beta : u\fl v$, and will be denoted by $w\beta w'$. One step reductions form the \emph{reduction relation} on $X^\ast$ denoted by $\fl_R$. A \emph{rewriting path} with respect to $R$ is a finite or infinite sequence $u_0 \fl_R u_1 \fl_R u_2 \fl_R \cdots $. This corresponds to the reflexive and transitive closure of the relation $\fl_R$, that we denote by $\overset{\ast}{\fl}_R$.
A word~$u$ in $X^\ast$ is \emph{reduced} if there is no reduction with source~$u$. A~\emph{normal form} for a word $u$ in $X^\ast$ is a reduced word~$v$ such that $u$ reduces into $v$. The rewriting system~$R$ \emph{terminates} if it has no infinite rewriting path, and it is \emph{(weakly) normalizing} if every word~$u$ in~$X^\ast$ reduces to some normal form.
 A rewriting system~$R$ is \emph{reduced} if, for every rule~$\beta:u\fl v$ in~$R$, the source~$u$ is~$(R\setminus \{\beta\})$-reduced and the target~$v$ is reduced.
The reflexive, symmetric and transitive closure of~$\fl_R$ is the congruence on $X^\ast$ generated by $R$, that we denote by $\approx_R$. The monoid presented by $R$ is the quotient of the free monoid~$X^\ast$ by the congruence~$\approx_R$.
Two rewriting systems are \emph{Tietze equivalent} if they present isomorphic monoids.  
Recall that a \emph{Tietze transformation} between two rewriting systems is a sequence of \emph{elementary Tietze transformations}, defined on a rewriting system $R$ on~$X$ by the following operations:
\begin{enumerate}[{\bf i)}]
\item adjunction or elimination of an element~$x$ in~$X$ and of a rule~$\beta:u\fl x$, where~$u$ is an element in~$X^\ast$ that does not contain~$x$,
\item adjunction or elimination of a rule~$\beta: u\fl v$ such that~$u$ and~$v$ are equivalent by the congruence generated by~$R\setminus \{\beta\}$.
\end{enumerate}
Two rewriting systems are Tietze equivalent if, and only if, there exists a Tietze transformation between them, see~\cite{GaussentGuiraudMalbos15} for more details.

\subsection{Confluence}
A \emph{branching} (resp. \emph{local branching}) of a rewriting system~$R$ on~$X$ is a non ordered pair~$(f,g)$ of reductions (resp. \emph{one step reductions}) of~$R$ on the same word. A branching is \emph{aspherical} if it is of the form~$(f,f)$, for a one step reduction~$f$ and \emph{Peiffer} when it is of the form~$(fv,ug)$ for one step reductions~$f$ and~$g$ with source $u$ and $v$ respectively. The \emph{overlapping} branchings are the remaining local branchings.
An overlapping local branching is \emph{critical} when it is minimal for the order~$\sqsubseteq$ generated by the relations~\mbox{$(f,g) \:\sqsubseteq\: \big( w f w', w g w')$},
given for all local branching~$(f,g)$ and words~$w,w'$ in~$X^\ast$.
A branching~$(f,g)$  is \emph{confluent} if there exist reductions~$f'$ and~$g'$ reducing to the same word:
\begin{equation}
\raisebox{0.6cm}{
\xymatrix @C=2.5em @R=0.2em {
& v
	\ar @{.>} @/^1ex/ [dr] ^{f'}
\\
u
	\ar @/^1ex/ [ur] ^{f}
	\ar @/_1ex/ [dr] _{g}
& &
w
\\
& v'
\ar @{.>}@/_1ex/ [ur] _{g'}
}}
\end{equation}
The rewriting system $R$ is \emph{confluent} if all of its branchings  are confluent, and \emph{convergent} if it is both confluent and terminating. If $R$ is convergent, then every word~$u$ in~$X^\ast$ has a unique normal form.

\subsection{Normalization strategies}
Recall that a \emph{reduction strategy} for a rewriting system $R$ on $X$ specifies a way to apply the rules in a deterministic way. It is defined as a mapping $\vartheta$ of every word $u$ in $X^\ast$ to a one step reduction~$\vartheta_u$ with source $u$.
When $R$ is normalizing, a \emph{normalization strategy} is a mapping~$\sigma$ of every word $u$ to a rewriting path $\sigma_u$ with source $u$ and target a  chosen normal form of $u$. For a reduced rewriting system, we distinguish  two canonical reduction strategies to reduce words: the leftmost one and the rightmost one, according to the way we apply first the rewriting rule that reduces the leftmost or the rightmost subword. They are defined as follows. For every word~$u$ of~$X^\ast$, the set of one step reductions with source~$u$ can be ordered from left to right by setting~$f\prec g$, for one step reductions~$f=v\gamma v'$ and~$g=w\beta w'$ such that~$|v| <|w|$. If~$R$ is finite, then the order~$\prec$ is total and the set of one step reductions of source~$u$ is finite. Hence this set contains a smallest element~$\rho_u$ and a greatest element~$\eta_u$, respectively called the \emph{leftmost} and the \emph{rightmost one step reductions on~$u$}. If, moreover, the rewriting system terminates, the iteration of~$\rho$ (resp.~$\eta$) yields a normalization strategy for~$R$ called the \emph{leftmost} (resp.\ \emph{rightmost}) \emph{normalization strategy of~$R$}: 
\begin{equation}
\sigma^\vdash_{u} := \rho_{u} \star_1 \sigma^\vdash_{t(\rho_u)}
\qquad
(\text{resp.\ }
\sigma^\dashv_u := \eta_u \star_1 \sigma^\dashv_{t(\eta_u)}).
\end{equation}
The \emph{leftmost} (resp. \emph{rightmost}) \emph{rewriting path} on a word $u$ is the rewriting path obtained by applying the leftmost (resp. rightmost) normalization strategy $\sigma^\vdash_u$ (resp. $\sigma^\dashv_u$). We refer the reader to~\cite{GuiraudMalbos12advances} for more details on rewriting normalization strategies.

A rewriting system~$R$ on $X$ is \emph{semi-quadratic}  if for all~$\gamma$ in $R$ we  have~$|s(\gamma)|=2$ and $|t(\gamma)|\leq 2$.
The sources of the critical branchings of a semi-quadratic rewriting system are of length~$3$. When $R$ is reduced, there are at most two rewriting paths with respect to $R$ with source a word of length~$3$.
We will denote by~$\len_l(w)$ (resp.~$\len_r(w)$) the length of the leftmost~(resp. rightmost) rewriting path from~$w$ to its normal form.

\subsection{Coherent presentations}
\label{SS:CoherentPresentation}

We recall the notion of coherent presentation of monoids formulated in terms of polygraphs in~\cite{GaussentGuiraudMalbos15}, see also~\cite{GuiraudMalbos18}.
Rewriting systems can be interpreted as $2$-polygraphs with only one $0$-cell. Such a \emph{$2$-polygraph} $P$ is a data~$(P_1,P_2)$, where~$P_1$ is a set and $P_2$ is a \emph{globular extension} of the free monoid~$P_1^\ast$ seen as a $1$-category. The elements of~$P_2$ are \emph{generating $2$-cells}~$\beta:u\dfl v$ relating $1$-cells in~$P_1^\ast$, with \emph{source}~$u$ and \emph{target}~$v$, denoted respectively by~$s_1(\beta)$  and~$t_1(\beta)$. 
A rewriting system~$R$ on~$X$ can be described by such a $2$-polygraph where the generating $2$-cells are the rules of~$R$. 
Recall that a \emph{$(2,1)$-category} is a category enriched in groupoids.  We will denote by~$\tck{P}_2$ the $(2,1)$-category freely generated by the $2$-polygraph $P$, see~\cite{GuiraudMalbos18} for expanded definitions.

A pair $(f,g)$ of $2$-cells of $\tck{P}_2$ such that  
$s_1(f)=s_1(g)$ and~$t_1(f)=t_1(g)$ is called a \emph{$2$-sphere} of $\tck{P}_2$. A \emph{$(3,1)$-polygraph} is a data~$(P,P_3)$ made of a $2$-polygraph $P$ and a globular extension $P_3$ of the $(2,1)$-category $\tck{P}_2$, that is a set of $3$-cells $A : f \tfl g$, where~$(f,g)$ is a $2$-sphere of~$\tck{P}_2$. The $2$-cell~$f$ (resp.~$g$) is called the source (resp. target) of~$A$, and denoted by $s_2(A)$ (resp.~$t_2(A)$). Such a $3$-cell can be represented with the following globular shape:
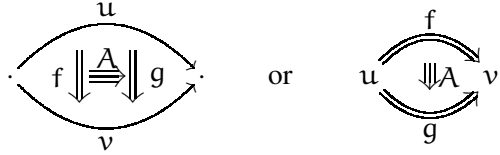
\[
\xymatrix @C=7em {
\Cdot
	\ar @/^3ex/ [r] ^-{u} _-{}="src"
	\ar @/_3ex/ [r] _-{v} ^-{}="tgt"
	\ar@2 "src"!<-15pt,-10pt>;"tgt"!<-15pt,10pt> _-*+{f} ^-{}="srcA"
	\ar@2 "src"!<+15pt,-10pt>;"tgt"!<+15pt,10pt> ^-*+{g} _-{}="tgtA"
	\ar@3 "srcA"!<8pt,0pt> ; "tgtA"!<-8pt,0pt> ^-{A}
&
\Cdot
}
\qquad\text{or}\qquad
\xymatrix@!C@C=3em{
u
	\ar@2@/^3ex/ [r] ^{f} _{}="src"
	\ar@2@/_3ex/ [r] _{g} ^{}="tgt"
& 
v
\ar@3 "src"!<0pt,-10pt>;"tgt"!<0pt,10pt> ^-{A}
}
\]
where $\cdot$ denotes the unique $0$-cell of $P$.
We will denote by $\tck{P}_3$ the free $(3,1)$-category generated by the $(3,1)$-polygraph $(P,P_3)$.
An \emph{extended presentation} of a monoid $\M$ is a  $(3,1)$-polygraph whose underlying $2$-polygraph is a presentation of $\M$. A \emph{coherent presentation of $\M$} is an extended presentation $(P, P_3)$ of~$\M$ such that the cellular extension $P_3$ of the $(2,1)$-category $\tck{P}_2$ is acyclic, that is, for every $2$-sphere~$(f,g)$ of~$\tck{P}_2$, there exists a $3$-cell $A$ in the $(3,1)$-category~$\tck{P}_3$ such that $s_2(A)=f$ and $t_2(A)=g$. The elements in~$\tck{P}_3$ are called \emph{syzygies} of the presentation~$P$.

Recall  Squier's coherence theorem from~\cite{Squier94}, see also~\cite{GuiraudMalbos18}, that  states that, any convergent rewriting system $R$ on $X$ presenting a monoid~$\M$ can be extended into a coherent presentation of~$\M$  having a generating syzygy
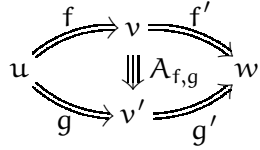
\[      
\xymatrix @R=0.4em @C=2.5em @!C {
        & v
        \ar@2@/^/  [dr] ^-{f'}
        \ar @3[]!<0pt,-10pt>;[dd]!<0pt,10pt> ^-*+{A_{f,g}}
        \\
        u
        \ar@2@/^/  [ur] ^-{f}
         \ar@2@/_/ [dr] _-{g}
        && w
        \\
        & v'
         \ar@2@/_/ [ur] _-{g'}
}
\]
for every critical branching~$(f,g)$ of~$R$, where~$f'$ and~$g'$ are chosen confluent rewriting paths.

\section{Insertions on Chinese staircases}
\label{S:CombinatoricsOfChineseStaircases}

In this section, we recall  the structure of Chinese staircase and the right and left insertion algorithms on Chinese staircases. 
We also recall the structure of Chinese monoid and the cross-section property for this monoid and we deduce properties of the insertions products on Chinese staircases.

\subsection{Chinese staircases}
\label{SS:Chinesestaircases}
A \emph{mirror Young diagram} of shape $(1,2,\ldots,n)$ is a collection of boxes in right-justified rows, whose rows (resp. columns) are indexed with  the totally ordered set~\mbox{$[n]:=\{1<\ldots<n\}$,} for $n$ in $\Zb_{>0}$, from top to bottom (resp. from right to left) and where the $i$-th row contains~$i$ boxes for~$1 \leq i \leq n$.
A \emph{(Chinese) staircase} over~$[n]$ is a mirror Young diagram of shape $(1,2,\ldots,n)$  filled with non-negative integers. Denote by~$t_{ij}$ (resp.~$t_i$) the contents of the box in row~$i$ and column~$j$ for~$i>j$ (resp. $i=j$). A box filled by~$0$ is called \emph{empty}.
Denote by~$\ChinT_n$ the set of staircases over~$[n]$ and by~$\R: \ChinT_n\to [n]^{\ast}$ the map that reads a staircase row by row, from right to left and from top to bottom, and where the $i$-th row is read as follows $(i1)^{t_{i1}}(i2)^{t_{i2}}\ldots(i(i-1))^{t_{i(i-1)}}(i)^{t_{i}}$, for~$1 \leq i \leq n$.
For instance, for the following staircase~$t$ over~$[4]$:
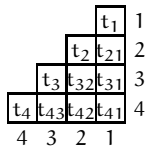
\begin{center}
\raisebox{+1cm}{
\ytableausetup{smalltableaux, boxsize=0.4cm}
\begin{ytableau}
\none & \none&\none& \scriptstyle t_1&  \none[\scriptstyle  1]\\
\none&\none& \scriptstyle t_2& \scriptstyle t_{21}& \none[\scriptstyle 2]\\
\none&\scriptstyle  t_3 & \scriptstyle t_{32}&\scriptstyle t_{31}& \none[\scriptstyle 3]\\
 \scriptstyle t_4& \scriptstyle t_{43}& \scriptstyle t_{42}& \scriptstyle t_{41}& \none[\scriptstyle 4]\\
\none[\scriptstyle 4]&\none[\scriptstyle 3]&\none[\scriptstyle 2]&\none[\scriptstyle 1]
\end{ytableau}
}
\end{center} 
we have~$\R(t)= 1^{t_1}(21)^{t_{21}}(2)^{t_2}(31)^{t_{31}}(32)^{t_{32}}(3)^{t_3}(41)^{t_{41}}(42)^{t_{42}}(43)^{t_{43}}(4)^{t_4}$.
By removing the bottom row of a staircase~$t$ over~$[n]$, we obtain a staircase over $[n-1]$, denoted by $t'$, as on the following picture:
\begin{center}
\raisebox{1cm}{$t\;=\;$}
\begin{tikzpicture}[scale=0.9]
\draw (-0.15,0.5)   -- (3,0.5);
\draw (-0.15,1)   -- (3,1);
\draw (0.5,1.5)   -- (1,1.5);
\draw[dashed] (1,2.5) -- (2.5,2.5);
\draw (2.5,3) -- (3,3);
\draw (-0.15,0.5)   -- (-0.15,1);
\draw (0.5,0.5) -- (0.5,1.5);
\draw[dashed] (1,1.5) -- (1,2.5);
\draw (2.5,2.5) -- (2.5,3);
\draw (3,0.5) -- (3,3);
\draw (2.5,0.5) -- (2.5,1);
\draw (1.25,0.4) node {};
\draw (3.2,2.8) node {$\scriptstyle 1$};
\draw (3.3,1.2) node {$\scriptstyle n-1$};
\draw (3.2,2) node {$\scriptstyle \vdots$};
\draw (3.2,0.6) node {$\scriptstyle n$};
\draw (0.2,0.7) node {$\scriptstyle t_{n}$};
\draw (2.75,0.7) node {$\scriptstyle t_{n1}$};
\draw (0.2,0.2) node {$\scriptstyle n$};
\draw (2.7,0.2) node {$\scriptstyle 1$};
\draw (1.7,0.2) node {$ \ldots$};
\draw (1.7,0.7) node {$ \ldots$};
\draw (1.7,1.7) node {$t'$};
\end{tikzpicture}
\end{center}
 According to this, such a staircase can be denoted by~$(t',R_1)$, where~$R_1$ is the bottom row of~$t$.

\subsection{The right insertion algorithm}
\label{SSS:Rightinsertion}
Recall the right insertion map~$\insr{\Ich} : \ChinT_n \times [n] \fl \ChinT_n$ introduced in~\cite{CassaigneEspieKrobNovelliHibert01}. Let~$t$ be a staircase and~$x$ an element in~$[n]$. If~$x=n$, then~\mbox{$t\insr{\Ich}x  = (t', R'_1)$,} where~$R'_1$ is obtained from~$R_1$ by adding~$1$ to~$t_n$. If~$x<n$, let~$y_1$ be maximal such that the entry in column~$y_1$ of $R_1$ is non-zero or if such a~$y_1$ does not exist, set~$y_1 = x$. Three cases appear:
\begin{enumerate}[\bf i)]
\item If~$x \geq y_1$, then~$t\insr{\Ich} x= (t'\insr{\Ich}x, R_1)$,
\item If~$x < y_1 < n$, then~$t\insr{\Ich}x =(t'\insr{\Ich}y_1, R'_1)$, where~$R'_1$ is obtained from~$R_1$ by subtracting~$1$ from~$t_{ny_{1}}$ and adding~$1$ to~$t_{nx}$,
\item If~$x < y_1 = n$, then~$t\insr{\Ich}x = (t', R'_1)$, where~$R'_1$ is obtained from $R_1$ by subtracting $1$ from $t_n$ and adding $1$ to~$t_{nx}$.
\end{enumerate}

For example, we compute 
$\big(\scalebox{1}{\raisebox{0.4cm}{
\ytableausetup{smalltableaux}
\begin{ytableau}
\none & \none&\none&\scriptstyle 1&\none[\scriptstyle 1]\\
\none&\none &\scriptstyle 1 &\scriptstyle 0 & \none[\scriptstyle 2]\\
\none&\scriptstyle 0 &\scriptstyle 1&\scriptstyle 1& \none[\scriptstyle 3]\\
\scriptstyle 0 &\scriptstyle 0 & \scriptstyle 2 &  \scriptstyle 0 & \none[\scriptstyle 4]\\
\none[\scriptstyle 4]&\none[\scriptstyle 3]&\none[\scriptstyle 2]&\none[\scriptstyle 1]
\end{ytableau}}}\insr{\Ich} 1\big)$
in three steps:
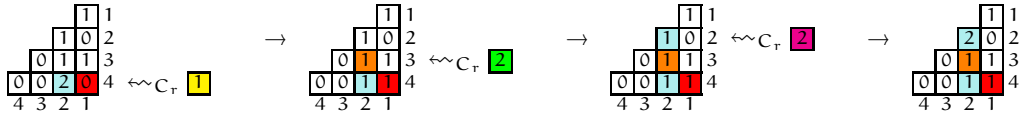
\[
\scalebox{1}{
\ytableausetup{smalltableaux}
\begin{ytableau}
\none &\none & \none&\none&\scriptstyle 1&\none[\scriptstyle 1]\\
\none &\none&\none &\scriptstyle 1 &\scriptstyle 0 & \none[\scriptstyle 2]\\
\none &\none&\scriptstyle 0 &\scriptstyle 1&\scriptstyle 1& \none[\scriptstyle 3]\\
\none &\scriptstyle 0 &\scriptstyle 0 & *(cyan)\scriptstyle 2 & *(red) \scriptstyle 0 & \none[\scriptstyle 4]&\none&\none[\insr{\Ich}]&\none &*(yellow) \scriptstyle 1\\
\none&\none[\scriptstyle 4]&\none[\scriptstyle 3]&\none[\scriptstyle 2]&\none[\scriptstyle 1]
\end{ytableau}
}\qquad
\underset{\fl}
\;
\scalebox{1}{
\ytableausetup{smalltableaux}
\begin{ytableau}
\none &\none & \none&\none&\scriptstyle 1&\none[\scriptstyle 1]\\
\none &\none&\none &\scriptstyle 1 &\scriptstyle 0 & \none[\scriptstyle 2]\\
\none &\none&\scriptstyle 0 & *(orange)\scriptstyle 1&\scriptstyle 1& \none[\scriptstyle 3]&\none&\none[\insr{\Ich}]&\none& *(green)\scriptstyle 2\\\none &\scriptstyle 0 &\scriptstyle 0 & *(cyan)\scriptstyle 1 & *(red) \scriptstyle 1 & \none[\scriptstyle  4]\\
\none&\none[\scriptstyle 4]&\none[\scriptstyle 3]&\none[\scriptstyle 2]&\none[\scriptstyle 1]
\end{ytableau}}
\qquad
\underset{\fl}
\;
\scalebox{1}{
\ytableausetup{smalltableaux}
\begin{ytableau}
\none &\none & \none&\none&\scriptstyle 1&\none[\scriptstyle 1]\\
\none &\none&\none & *(cyan)\scriptstyle 1 &\scriptstyle 0 & \none[\scriptstyle 2]&\none&\none[\insr{\Ich}]&\none& *(magenta)\scriptstyle 2\\
\none &\none&\scriptstyle 0 & *(orange)\scriptstyle 1&\scriptstyle 1& \none[\scriptstyle 3]\\
\none &\scriptstyle 0 &\scriptstyle 0 & *(cyan)\scriptstyle 1 & *(red) \scriptstyle 1 & \none[\scriptstyle 4]\\
\none&\none[\scriptstyle 4]&\none[\scriptstyle 3]&\none[\scriptstyle 2]&\none[\scriptstyle 1]
\end{ytableau}}
\qquad
\underset{\fl}
\;
\scalebox{1}{
\ytableausetup{smalltableaux}
\begin{ytableau}
\none &\none & \none&\none& 1&\none[\scriptstyle 1]\\
\none &\none&\none & *(cyan)\scriptstyle 2 &\scriptstyle 0 & \none[\scriptstyle 2]\\
\none &\none&\scriptstyle 0 & *(orange)\scriptstyle 1&\scriptstyle 1& \none[\scriptstyle 3]\\
\none &\scriptstyle 0 &\scriptstyle 0 & *(cyan)\scriptstyle 1 & *(red) \scriptstyle 1 & \none[\scriptstyle 4]\\
\none&\none[\scriptstyle 4]&\none[\scriptstyle 3]&\none[\scriptstyle 2]&\none[\scriptstyle 1]
\end{ytableau}}
\]

\subsection{The left insertion algorithm}
\label{SSS:leftInsertion}
A left insertion map~$\insl{\Jch} : \ChinT_n \times [n] \fl \ChinT_n$  that inserts an element~$x$ in~$[n]$ into a staircase~$t$, is defined in~\cite{CainGrayMalheiroChinese} in two steps as follows. Let~$y$ be an element in~$[n] \cup \{ \lambda \}$, initially set to~$\lambda$.

{\bf Step 1.} For~$i= 1,\ldots, x-1$, iterate the following. If every entry in the $i$-th row is empty, do nothing. Otherwise, let~$z$ be minimal such that~$t_{iz}$ is non-zero. There are two cases according to the values of~$y$:
\begin{enumerate}[\bf i)]
\item Suppose~$y = \lambda$. If~$z< i$, decrement~$t_{iz}$ by~$1$, increment~$t_{i}$ by~$1$, and set~$y = z$. If~$z = i$, decrement~$t_i$ by~$1$, and set~$y = z$.
\item Suppose~$y \neq \lambda$. If~$z< y$, decrement~$t_{iz}$ by~$1$, increment~$t_{iy}$ by~$1$, and set~$y = z$. If $z\geq y$, do nothing.
\end{enumerate}

{\bf Step 2.} For~$i = x$, if~$y = \lambda$, then increment~$t_{i}$ by~$1$. Otherwise, decrement~$t_{iy}$ by~$1$. 

For example, we compute 
$\big(4 \insl{\Jch} \scalebox{1}{\raisebox{0.4cm}{
\ytableausetup{smalltableaux}
\begin{ytableau}
\none &\none&\none&\scriptstyle 0 &\none[\scriptstyle 1]\\
\none&\none &\scriptstyle 1 &\scriptstyle 0 & \none[\scriptstyle 2]\\
\none&\scriptstyle 0 &\scriptstyle 1&\scriptstyle 1& \none[\scriptstyle 3]\\
\scriptstyle 0 &\scriptstyle 0 &\scriptstyle  2 &\scriptstyle  0 & \none[\scriptstyle 4]\\
\none[\scriptstyle 4]&\none[\scriptstyle 3]&\none[\scriptstyle 2]&\none[\scriptstyle 1]
\end{ytableau}}
}\big)$ in three steps:
\[
\scalebox{1}{
\ytableausetup{smalltableaux}
\begin{ytableau}
*(yellow)\scriptstyle 4&\none&\none[\insl{\Ich}] &\none &\none&\none&\scriptstyle 0 &\none[\scriptstyle 1]\\
\none &\none&\none &\none&\none &\scriptstyle 1 &\scriptstyle 0 & \none[\scriptstyle 2]\\
\none &\none&\none &\none&\scriptstyle 0 &\scriptstyle 1&\scriptstyle 1& \none[\scriptstyle 3]\\
\none &\none&\none &\scriptstyle 0 &\scriptstyle 0 &\scriptstyle  2 &\scriptstyle  0 & \none[\scriptstyle 4]\\
\none &\none&\none&\none[\scriptstyle 4]&\none[\scriptstyle 3]&\none[\scriptstyle 2]&\none[\scriptstyle 1]
\end{ytableau}}
\qquad
\underset{\fl}
\qquad
\scalebox{1}{
\ytableausetup{smalltableaux}
\begin{ytableau}
\none &\none & \none&\none&\scriptstyle 0 &\none[\scriptstyle 1]\\
\none &\none&\none & *(red)\scriptstyle0 &\scriptstyle 0 & \none[\scriptstyle 2]\\
\none &\none&\scriptstyle 0 &\scriptstyle  1&\scriptstyle 1& \none[\scriptstyle 3]\\
\none &\scriptstyle 0 &\scriptstyle 0 &\scriptstyle  1 &\scriptstyle  1 & \none[\scriptstyle 4]\\
\none&\none[\scriptstyle 4]&\none[\scriptstyle 3]&\none[\scriptstyle 2]&\none[\scriptstyle 1]
\end{ytableau}}
\qquad
\underset{\fl}
\qquad
\scalebox{1}{
\ytableausetup{smalltableaux}
\begin{ytableau}
\none &\none & \none&\none& 0&\none[\scriptstyle 1]\\
\none &\none&\none & *(red) \scriptstyle 1 &\scriptstyle 0 & \none[\scriptstyle 2]\\
\none &\none&\scriptstyle 0 & *(orange)\scriptstyle 2& *(cyan)\scriptstyle 0& \none[\scriptstyle 3]\\
\none &\scriptstyle 0 &\scriptstyle 0 &\scriptstyle  1 &\scriptstyle  1 & \none[\scriptstyle 4]\\
\none&\none[\scriptstyle 4]&\none[\scriptstyle 3]&\none[\scriptstyle 2]&\none[\scriptstyle 1]
\end{ytableau}}
\qquad
\underset{\fl}
\qquad
\scalebox{1}{
\ytableausetup{smalltableaux}
\begin{ytableau}
\none &\none & \none& \none&\scriptstyle 0&\none[\scriptstyle 1]\\
\none &\none& \none & *(red)\scriptstyle 1 &\scriptstyle 0 & \none[\scriptstyle 2]\\
\none &\none&\scriptstyle 0 & *(orange)\scriptstyle 2& *(cyan)\scriptstyle 0& \none[\scriptstyle 3]\\
\none &\scriptstyle 0 & \scriptstyle 0 &\scriptstyle  1 & *(magenta)\scriptstyle 1 & \none[\scriptstyle 4]\\
\none&\none[\scriptstyle 4]&\none[\scriptstyle 3]&\none[\scriptstyle 2]&\none[\scriptstyle 1]
\end{ytableau}}
\]

\subsection{Insertion products on Chinese staircases}
\label{SSS:InsertionProductChineseStaircases}

For any word $w=x_1\ldots x_k$, denote by~$\lecCol{w}_{\Ich}$ (resp.~$\lecCol{w}_{\Jch}$) the staircase obtained from~$w$ by inserting its letters iteratively from left to right (resp. right to left) using the right (resp. left) insertion starting from the empty staircase:
\[
\begin{array}{rl}
\lecCol{w}_{\Ich}
\;
:=&
\;
(\emptyset \insr{\Ich} w)
\;
=
\;((\ldots(\emptyset \insr{\Ich} x_1)  \insr{\Ich} \ldots )\insr{\Ich} x_k),\\
\big(\text{resp. }\lecCol{w}_{\Jch} 
\;
:=&
\;
(w\insl{\Jch} \emptyset )
\;
=
(x_1 \insl{\Jch} ( \ldots\insl{\Jch}  (x_k\insl{\Jch} \emptyset)\ldots))\big).
\end{array}
\]
Define now an internal product~$\star_{\Ich}$ (resp.~$\star_{\Jch}$) on $ \ChinT_n$ by setting 
\begin{equation}
\label{Eq:StructureMonoidProduct}
t \star_{\Ich} t' :=  (t\insr{\Ich} \R(t')),
\qquad
\big(\text{resp. }  t \star_{\Jch} t' :=  (\R(t')\insl{\Jch}t)\big)
\end{equation}
for all $t,t'$ in $\ChinT_n$.
By definition the relations $t\star_{\Ich} \emptyset = t$ (resp. $t\star_{\Jch} \emptyset = t$) and $\emptyset \star_{\Ich} t = t$ (resp.~$\emptyset \star_{\Jch} t = t$) hold, showing that the product~$\star_{\Ich}$ (resp.~$\star_{\Jch}$) is unitary with respect to~$\emptyset$.

\subsection{The cross-section property}
\label{SSS:Chinesemonoids}
The \emph{Chinese monoid of rank $n>0$},  denoted by~$\C_n$, is presented by the rewriting system on~$[n]$, whose rules are the \emph{Chinese relations},~\cite{DuchampKrob94}: 
\begin{equation}
\label{ChineseRelations}
\begin{array}{rl}
& zyx \fl yzx \quad \text{ and } \quad zxy \fl yzx \quad \text{for all} \quad 1\leq x< y < z \leq n,\\
&   yyx  \fl yxy\quad \text{ and } \quad  yxx \fl xyx  \quad \text{for all} \quad 1\leq x< y \leq n.
\end{array}
\end{equation}

These relations generate the \emph{Chinese congruence}, denoted by~$\sim_{\C_n}$, which can be also interpreted in terms of Chinese staircases as follows. The set of Chinese staircases satisfies the \emph{cross-section property} for the monoid~$\C_n$, 
that is, for all words $w,w'$ on $[n]$, $w\sim_{\C_n} w'$ if and only if~$\lecCol{w}_{\Ich} = \lecCol{w'}_{\Ich}$,~{\cite[Theorem 2.1]{CassaigneEspieKrobNovelliHibert01}}.
As a consequence of the cross-section property, we deduce the following result.

\begin{corollary}
\label{C:oppositeSDSAssociativitycondition}
The composition $\star_{\Ich}$ is associative and the following equality
\begin{equation}
\label{E:CommutationChinese}
y \insl{\Jch} (t \insr{\Ich} x )  
\:=\:
(y \insl{\Jch} t) \insr{\Ich} x
\end{equation}
 holds in~$\ChinT_n$, for all~$t$ in~$\ChinT_n$ and $x,y$ in $[n]$. In particular, the composition~$\star_{\Jch}$ is associative and the following relation
\begin{equation}
\label{OppositeEqu}
t\star_{\Ich} t' = t'\star_{\Jch} t
\end{equation}
holds for all $t,t'$ in~$\ChinT_n$. 
\end{corollary}

\section{Column presentation of the Chinese monoid}
\label{S:SemiQuadraticConvergentPresentationChineseMonoids}

We construct a finite semi-quadratic convergent presentation of the Chinese monoid~$\C_n$ by adding the columns over~$[n]$ of length at most~$2$ and square generators to the presentation~(\ref{ChineseRelations}) and by using the combinatorial properties of the insertion algorithms on the Chinese staircases.

\subsection{Column generators}
\label{SSS:columngenerators}
We consider one \emph{column generator~$c_{yx}$ of length $2$} for all~$1\leq x<y\leq n$, one \emph{column generator~$c_{x}$ of length $1$} for any $1\leq x\leq n$, and one \emph{square generator~$c_{xx}$} for any~$1 < x< n$, corresponding to the following three staircases:
\[
\scalebox{1}{
\begin{tikzpicture}[scale=0.7]
\draw (-0.15,0.5)   -- (4,0.5);
\draw (-0.15,1)   -- (4,1);
\draw[dashed] (0.5,1.5)   -- (1,1.5);
\draw (3.5,4)   -- (4,4);
\draw[dashed] (2,2.5) -- (4,2.5);
\draw[dashed] (2.5,3) -- (4,3);
\draw[dashed] (1,1.7) -- (4,1.7);
\draw[dashed] (1,2.1) -- (4,2.1);
\draw[dashed] (3.3,3.5) -- (3.5,3.5);
\draw (-0.15,0.5)   -- (-0.15,1);
\draw[dashed] (0.5,0.5) -- (0.5,1.5);
\draw[dashed] (2.5,0.5) -- (2.5,3);
\draw[dashed] (3.3,0.5) -- (3.3,3.5);
\draw (2.5,0.5) -- (2.5,1);
\draw (4,0.5) -- (4,4);
\draw[dashed] (1,0.5) -- (1,2.1);
\draw[dashed] (2,0.5) -- (2, 2.5);
\draw[dashed] (1,0.5) -- (1,1);
\draw (3.5,3.5) -- (3.5,4);
\draw (4.2,3.8) node {$\scriptstyle 1$};
\draw (4.2,2.8) node {$\scriptstyle x$};
\draw (4.2,2.6) node {$\scriptstyle \vdots$};
\draw (4.2,1.85) node {$\scriptstyle y$};
\draw (2.9,1.85) node {$\scriptstyle 1$};
\draw (4.2,1.5) node {$\scriptstyle \vdots$};
\draw (4.2,0.6) node {$\scriptstyle n$};
\draw (0.2,0.2) node {$\scriptstyle n$};
\draw (2.8,0.2) node {$\scriptstyle x$};
\draw (3.9,0.2) node {$\scriptstyle 1$};
\draw (1.5,0.2) node {$\scriptstyle y$};
\draw (2.1,0.2) node {$\scriptstyle \ldots$};
\draw (0.8,0.2) node {$\scriptstyle \ldots$};
\draw (3.3,0.2) node {$\scriptstyle \ldots$};
\draw (0.8,0.7) node {$\scriptstyle \ldots$};
\draw (2.2,0.7) node {$\scriptstyle \ldots$};
\draw (2.2,1.3) node {$\scriptstyle \ldots$};
\draw (2.2,2.2) node {$\scriptstyle \ldots$};
\draw (2.2,1.9) node {$\scriptstyle \ldots$};
\draw (3.6,0.7) node {$\scriptstyle \ldots$};
\draw (3.6,1.3) node {$\scriptstyle \ldots$};
\draw (3.6,2.2) node {$\scriptstyle \ldots$};
\draw (3.6,1.9) node {$\scriptstyle \ldots$};
\draw (3.6,3.2) node {$\scriptstyle \ldots$};
\draw (3.6,2.6) node {$\scriptstyle \ldots$};
\draw (3,2.7) node {$\scriptstyle t_{x}$};
\draw (4.2,3.5) node {$\scriptstyle \vdots$};
\draw[fill=cyan]  (3.3,0.5)-- (3.3,1)-- (4,1)-- (4,0.5) -- cycle;
\draw[dashed, fill=cyan]  (3.3,1)-- (3.3,1.7)-- (4,1.7)-- (4,1) -- cycle;
\draw[dashed, fill=cyan]  (3.3,1.7)-- (3.3,2.1)-- (4,2.1)-- (4,1.7) -- cycle;
\draw[dashed, fill=cyan]  (3.3,2.1)-- (3.3,2.5)-- (4,2.5)-- (4,2.1) -- cycle;
\draw[dashed, fill=cyan]  (3.3,2.5)-- (3.3,3)-- (4,3)-- (4,2.5) -- cycle;
\draw[dashed, fill=cyan]  (3.3,3)-- (3.3,3.5)--(3.5,3.5)--(3.5,4)-- (4,4)-- (4,3) -- cycle;
\draw[ fill=cyan]  (2.5,0.5)-- (2.5,1) --(3.3,1)-- (3.3,0.5)--  cycle;
\draw[dashed, fill=cyan]  (2.5,1)-- (2.5,1.7) --(3.3,1.7)-- (3.3,1)--  cycle;
\draw[dashed, fill=cyan]  (2.5,2.1)-- (2.5,2.5)-- (3.3,2.5)-- (3.3,2.1) -- cycle;
\draw[dashed, fill=cyan]  (2.5,2.5)-- (2.5,3)-- (3.3,3)-- (3.3,2.5) -- cycle;
\draw[fill=cyan]  (2,0.5)-- (2,1) --(2.5,1)-- (2.5,0.5)--  cycle;
\draw[dashed, fill=cyan]  (2,1)-- (2,1.7) --(2.5,1.7)-- (2.5,1)--  cycle;
\draw[dashed, fill=cyan]  (2,1.7)-- (2,2.1) --(2.5,2.1)-- (2.5,1.7)--  cycle;
\draw[dashed, fill=cyan]  (2,2.1)-- (2,2.5)-- (2.5,2.5)-- (2.5,2.1) -- cycle;
\draw[ fill=cyan]  (1,0.5)-- (1,1) --(2,1)-- (2,0.5)--  cycle;
\draw[dashed, fill=cyan]  (1,1)-- (1,1.7) --(2,1.7)-- (2,1)--  cycle;
\draw[dashed, fill=cyan]  (1,1.7)-- (1,2.1) --(2,2.1)-- (2,1.7)--  cycle;
\draw[ fill=cyan]  (0.5,0.5)-- (0.5,1) --(1,1)-- (1,0.5)--  cycle;
\draw[dashed, fill=cyan]  (0.5,1)-- (0.5,1.6) --(1,1.6)-- (1,1)--  cycle;
\draw[dashed, fill=cyan]  (0.5,1)-- (0.5,1.6) --(1,1.6)-- (1,1)--  cycle;
\draw[ fill=cyan]  (-0.2,0.5)-- (-0.2,1) --(0.5,1)-- (0.5,0.5)--  cycle;
\end{tikzpicture}}
\qquad\qquad
\scalebox{1}{
\begin{tikzpicture}[scale=0.7]
\draw (-0.15,0.5)   -- (3,0.5);
\draw (-0.15,1)   -- (3,1);
\draw (0.5,1.5)   -- (1,1.5);
\draw[dashed] (1.7,2.5) -- (2.5,2.5);
\draw (2.5,3) -- (3,3);
\draw (-0.15,0.5)   -- (-0.15,1);
\draw (0.5,0.5) -- (0.5,1.5);
\draw[dashed] (1,1.5) -- (1,2.1);
\draw (2.5,2.5) -- (2.5,3);
\draw (3,0.5) -- (3,3);
\draw (2.5,0.5) -- (2.5,1);
\draw[dashed] (1,0.5) -- (1,2.1);
\draw[dashed] (1.7,0.5) -- (1.7,2.5);
\draw (1,0.5) -- (1,1);
\draw[dashed] (1,1.7) -- (3,1.7);
\draw[dashed] (1,2.1) -- (3,2.1);
\draw (1.25,0.4) node {};
\draw (3.2,2.8) node {$\scriptstyle 1$};
\draw (3.2,2.4) node {$\scriptstyle \vdots$};
\draw (3.2,1.85) node {$\scriptstyle x$};
\draw (3.2,1.5) node {$\scriptstyle \vdots$};
\draw (3.2,0.6) node {$\scriptstyle n$};
\draw (0.2,0.2) node {$\scriptstyle n$};
\draw (2.7,0.2) node {$\scriptstyle 1$};
\draw (1.3,0.2) node {$\scriptstyle x$};
\draw (2.1,0.2) node {$\scriptstyle \ldots$};
\draw (0.8,0.2) node {$\scriptstyle \ldots$};
\draw (0.8,0.7) node {$\scriptstyle \ldots$};
\draw (2.1,0.7) node {$\scriptstyle \ldots$};
\draw (1.4,1.9) node {$\scriptstyle  1$};
\draw[dashed, fill=cyan]  (1.7,1.7)-- (1.7,2.1)-- (3,2.1)-- (3,1.7) -- cycle;
\draw[dashed, fill=cyan]  (1.7,2.1)-- (1.7,2.5)-- (2.5,2.5)-- (2.5,3)-- (3,3)-- (3,2.1) -- cycle;
\draw[dashed, fill=cyan]  (1.7,1)-- (1.7,1.7)-- (3,1.7)-- (3,1) -- cycle;
\draw[dashed, fill=cyan]  (1,1)-- (1,1.7)-- (1.7,1.7)-- (1.7,1) -- cycle;
\draw[dashed, fill=cyan]  (0.5,1)-- (0.5,1.5)-- (1,1.5)-- (1,1) -- cycle;
\draw[dashed, fill=cyan]  (1.7,0.5)-- (1.7,1)-- (3,1)-- (3,0.5) -- cycle;
\draw[fill=cyan]  (1.7,0.5)-- (1.7,1)-- (3,1)-- (3,0.5) -- cycle;
\draw[fill=cyan]  (1,0.5)-- (1,1) --(1.7,1)-- (1.7,0.5)--  cycle;
\draw[fill=cyan]  (0.5,0.5)-- (0.5,1) --(1,1)-- (1,0.5)--  cycle;
\draw[fill=cyan]  (-0.2,0.5)-- (-0.2,1) --(0.5,1)-- (0.5,0.5)--  cycle;
\end{tikzpicture}
}
\qquad\qquad
\scalebox{1}{
\begin{tikzpicture}[scale=0.7]
\draw (-0.15,0.5)   -- (3,0.5);
\draw (-0.15,1)   -- (3,1);
\draw (0.5,1.5)   -- (1,1.5);
\draw[dashed] (1.7,2.5) -- (2.5,2.5);
\draw (2.5,3) -- (3,3);
\draw (-0.15,0.5)   -- (-0.15,1);
\draw (0.5,0.5) -- (0.5,1.5);
\draw[dashed] (1,1.5) -- (1,2.1);
\draw (2.5,2.5) -- (2.5,3);
\draw (3,0.5) -- (3,3);
\draw (2.5,0.5) -- (2.5,1);
\draw[dashed] (1,0.5) -- (1,2.1);
\draw[dashed] (1.7,0.5) -- (1.7,2.5);
\draw (1,0.5) -- (1,1);
\draw[dashed] (1,1.7) -- (3,1.7);
\draw[dashed] (1,2.1) -- (3,2.1);
\draw (1.25,0.4) node {};
\draw (3.2,2.8) node {$\scriptstyle 1$};
\draw (3.2,2.4) node {$\scriptstyle \vdots$};
\draw (3.2,1.85) node {$\scriptstyle x$};
\draw (3.2,1.5) node {$\scriptstyle \vdots$};
\draw (3.2,0.6) node {$\scriptstyle n$};
\draw (0.2,0.2) node {$\scriptstyle n$};
\draw (2.7,0.2) node {$\scriptstyle 1$};
\draw (1.3,0.2) node {$\scriptstyle x$};
\draw (2.1,0.2) node {$\scriptstyle \ldots$};
\draw (0.8,0.2) node {$\scriptstyle \ldots$};
\draw (0.8,0.7) node {$\scriptstyle \ldots$};
\draw (2.1,0.7) node {$\scriptstyle \ldots$};
\draw (1.4,1.9) node {$\scriptstyle  2$};
\draw[dashed,fill=cyan]  (1.7,1.7)-- (1.7,2.1)-- (3,2.1)-- (3,1.7) -- cycle;
\draw[dashed, fill=cyan]  (1.7,2.1)-- (1.7,2.5)-- (2.5,2.5)-- (2.5,3)-- (3,3)-- (3,2.1) -- cycle;
\draw[dashed,fill=cyan]  (1.7,1)-- (1.7,1.7)-- (3,1.7)-- (3,1) -- cycle;
\draw[dashed,fill=cyan]  (1,1)-- (1,1.7)-- (1.7,1.7)-- (1.7,1) -- cycle;
\draw[dashed,fill=cyan]  (0.5,1)-- (0.5,1.5)-- (1,1.5)-- (1,1) -- cycle;
\draw[dashed,fill=cyan]  (1.7,0.5)-- (1.7,1)-- (3,1)-- (3,0.5) -- cycle;
\draw[fill=cyan] (1.7,0.5)-- (1.7,1)-- (3,1)-- (3,0.5) -- cycle;
\draw[fill=cyan] (1,0.5)-- (1,1) --(1.7,1)-- (1.7,0.5)--  cycle;
\draw[fill=cyan] (0.5,0.5)-- (0.5,1) --(1,1)-- (1,0.5)--  cycle;
\draw[fill=cyan] (-0.2,0.5)-- (-0.2,1) --(0.5,1)-- (0.5,0.5)--  cycle;
\end{tikzpicture}}
\]
where the shaded areas represent empty boxes. We will denote by~$Q_n$ the set defined by
\[
Q_n\: := \:\big\{c_{yx} \; \big| \; 1\leq x<y\leq n \big\}\cup \big\{c_{xx} \; \big| \; 1 < x< n \big\} \cup \big\{c_1,\ldots, c_n \big\}.
\]

Let us define the map~$R_{Q_n}: \ChinT_n\to Q_n^{\ast}$ that reads a staircase row by row, from right to left and from top to bottom, and where the reading of the $i$-th row, for~$1 \leq i \leq n$,  is the following word in~$ Q_n^{\ast}$:
\[
\left\{
    \begin{array}{ll}
  \underbrace{c_{i1}\Cdot\ldots \Cdot c_{i1}}_\textrm{$t_{i1}$ times }\Cdot\underbrace{c_{i2}\Cdot\ldots \Cdot c_{i2}}_\textrm{$t_{i2}$ times }\Cdot\ldots\Cdot
c_{i}\Cdot\underbrace{c_{ii}\Cdot\ldots \Cdot c_{ii}}_\textrm{$\frac{1}{2} (t_{i}-1)$ times}         & \mbox{ when } t_{i}  \text{ is an odd number} \\
    \underbrace{c_{i1}\Cdot\ldots \Cdot c_{i1}}_\textrm{$t_{i1}$ times }\Cdot\underbrace{c_{i2}\Cdot\ldots \Cdot c_{i2}}_\textrm{$t_{i2}$ times }\Cdot\ldots\Cdot
\underbrace{c_{ii}\Cdot\ldots \Cdot c_{ii}}_\textrm{$\frac{1}{2} t_{i}$ times}     & \mbox{ when  } t_{i}  \text{ is an even  number.}
    \end{array}
\right.
\]
For instance,  consider  the following staircase over~$[4]$: 
\[
\raisebox{-0.5cm}{$t\;=\;$}
\ytableausetup{smalltableaux}
\begin{ytableau}
\none &\none&\none&\scriptstyle 1 &\none[\scriptstyle 1]\\
\none&\none &\scriptstyle 3 &\scriptstyle 0 & \none[\scriptstyle 2]\\
\none&\scriptstyle 0 &\scriptstyle 1&\scriptstyle 3& \none[\scriptstyle 3]\\
\scriptstyle 4 &\scriptstyle 0 &\scriptstyle  2 &\scriptstyle  1 & \none[\scriptstyle 4]\\
\none[\scriptstyle 4]&\none[\scriptstyle 3]&\none[\scriptstyle 2]&\none[\scriptstyle 1]
\end{ytableau}
\quad
\raisebox{-0.5cm}{$\text{ with }R_{Q_n}(t) = c_1\Cdot c_2\Cdot c_{22}\Cdot c_{31}\Cdot c_{31}\Cdot c_{31}\Cdot c_{32}\Cdot c_{41}\Cdot c_{42}\Cdot c_{42}\Cdot c_{44}\Cdot c_{44}$.}
\]

\subsection{Reduced column presentation} 
\label{SSS:reducedcolumnpresentation}
We denote by~$\Srs(Q_n, \C_n)$ the rewriting system on~$Q_n$ whose rules are 
\[
\gamma_{u,v} : c_u\Cdot c_v \fl R_{Q_n}(c_u\star_{\Ich} c_v)
\]
 for all~$c_u$ and~$c_v$ in~$Q_n$ such that~$c_u\Cdot c_v \neq R_{Q_n}(c_u\star_{\Ich} c_v)$. Normal forms with respect to this rewriting system are called \emph{Chinese normal forms}.
Note that the leftmost and rightmost reductions are the only reductions on a word $c_u\Cdot c_v\Cdot c_t$ in $Q_n^\ast$ with respect to~$\Srs(Q_n, \C_n)$. There will be denoted respectively by 
\begin{equation}
\label{E:ReductionsLeftRight}
\gamma_{\widehat{u,v},t}:=\gamma_{u,v}\Cdot c_t
\qquad
\text{ and }
\qquad
\gamma_{u,\widehat{v,t}}:=c_u\Cdot \gamma_{v,t}.
\end{equation}

\begin{theorem}
\label{Theorem:FSQCPChinese}
The rewriting system~$\Srs(Q_n, \C_n)$ is a finite semi-quadratic convergent presentation of the Chinese monoid~$\C_n$.
\end{theorem}

Theorem~\ref{Theorem:FSQCPChinese} will be proved later in the section. First, we deduce the following corollary:

\begin{corollary}
\label{Corollary:crossSectionProprty}
The following properties hold:
\begin{enumerate}[\bf i)]
\item The monoid~$\C_n$ has finite derivation type~$\mathrm{FDT}_{\infty}$. 
\item The monoid~$\C_n$ has finite homological type~$\mathrm{FP}_{\infty}$. 
\end{enumerate}
\end{corollary}

\begin{proof}
In~\cite{GuiraudMalbos12advances} the authors showed that if a monoid admits a finite convergent presentation, then it is of finite derivation type~$\mathrm{FDT}_{\infty}$, and the property of finite derivation type implies the property of finite homological type~$\mathrm{FP}_{\infty}$. Thus, Conditions~{\bf i)} and~{\bf ii)} are consequences of Theorem~\ref{Theorem:FSQCPChinese}.
\end{proof}

The rest of this section is devoted to the proof of Theorem~\ref{Theorem:FSQCPChinese}.
First, prove that~$\Srs(Q_n, \C_n)$ is a semi-quadratic presentation of the monoid~$\C_n$. We add in Subsection~\ref{SubSubsection:Precolumn presentation} the columns generators of length~$2$ and the square generators with their defining rules. This forms a non-confluent rewriting system that we complete into a presentation of~$\C_n$, that we call the \emph{precolumn presentation}. Then we show in Subsection~\ref{SSS:CompletionPrecolumnpresentation} that the rules of~$\Srs(Q_n, \C_n)$ are obtained from the precolumn presentation by applying one step of Knuth-Bendix's completion, \cite{KnuthBendix70}, on the precolumn presentation. Hence~$\Srs(Q_n, \C_n)$ is a presentation of the monoid~$\C_n$. Finally, we show in Proposition~\ref{Proposition:ConvergenceColumnPresentation} that~$\Srs(Q_n, \C_n)$ is terminating and  is confluent using the associativity of the product~$\star_{\Ich}$.

\subsection{Precolumn presentation}
\label{SubSubsection:Precolumn presentation}
Consider the rewriting system $\Chin_2(n)$ on~$\{c_{1},\ldots,c_{n}\}$ and whose rules are given by the following four families
\begin{equation}
\label{ChineseRelationsColumn}
\begin{array}{rl}
&\epsilon_{x,y,z}: c_z\Cdot c_y\Cdot c_x \fl c_y\Cdot c_z\Cdot c_x\;  \text{ and }\; \eta_{x,y,z}: c_z\Cdot c_x\Cdot c_y \fl c_y\Cdot c_z\Cdot c_x \quad \text{for all} \quad 1\leq x< y < z \leq n,\\
& \epsilon_{x,y}:  c_y\Cdot c_y\Cdot c_x  \fl c_y\Cdot c_x\Cdot c_y\; \text{ and }\;  \eta_{x,y}: c_y\Cdot c_x\Cdot c_x\fl c_x\Cdot c_y\Cdot c_x  \quad \text{for all} \quad 1\leq x< y \leq n,
\end{array}
\end{equation}
corresponding to the Chinese relations (\ref{ChineseRelations}), hence is a presentation of the monoid $\C_n$.
We add to the set of rules~(\ref{ChineseRelationsColumn}) the following set of rules
\[
\Gamma_2(n) =
\{\;\gamma_{y,x}:c_{y}\Cdot c_{x} \fl c_{yx} \;|\; 1\leq x< y  \leq n\;\} 
\cup
\{\;\gamma_{x,x}:c_{x}\Cdot c_{x} \fl c_{xx} \;|\; 1< x< n\;\},
\]
making a rewriting system $\Chin^{c}_2(n) =  \Gamma_2(n) \cup \Chin_2(n)$ on $Q_n$ that presents the monoid $\C_n$. 

\begin{lemma} 
\label{PrecolonnePresentation}
For~$n>0$, the rewriting system~$\PreColo_2(n)$ on~$Q_n$, whose set of rules is $\Gamma_2(n)\,\cup \, \Delta_2(n)$, where
\[
\begin{array}{rl}
\Delta_2(n) = &\,\{\;\gamma_{y,yx}: c_y\Cdot c_{yx}\fl c_{yx}\Cdot c_{y} \;\text{ for }\; 1\leq x<y\leq n\;\text{ and }\;\gamma_{yy,x} : c_{yy}\Cdot c_x \fl c_{yx}\Cdot c_y \;\text{ for }\; 1\leq x<y<n \} 
\\
& \,\cup \,\{\;\gamma_{zy,x}: c_{zy}\Cdot c_x\fl c_{y}\Cdot c_{zx} \; \text{ and } \;\gamma_{z,yx}: c_{z}\Cdot c_{yx}\fl c_{y}\Cdot c_{zx}
\;\text{ for }\; 1\leq x\leq y<z\leq n\;\}
\\
& \,\cup \,\{\;\gamma_{zx,y}: c_{zx}\Cdot c_y\fl c_{y}\Cdot c_{zx}\;\text{ for }\; 1\leq x< y<z\leq n\;\}
.
\end{array}
\]
 is a finite semi-quadratic presentation of the Chinese monoid~$\C_n$. 
\end{lemma}
\begin{proof}
We make explicit a Tietze equivalence between the rewriting systems~$\Chin^{c}_2(n)$ and~$\PreColo_2(n)$.
For~$1\leq x<y\leq n$, consider the following critical branching
\[
\small
\xymatrix @C=1em @R=0.5em {
& {c_{y}\Cdot c_{x}\Cdot c_{y}}
  \ar[r] ^-{\gamma_{\widehat{y,x},y}}
& {c_{yx}\Cdot c_{y}}
\\
{c_{y}\Cdot c_{y}\Cdot c_{x}}
	\ar@/^/ [ur] ^-{\epsilon_{x,y}}
	\ar@/_/ [dr] _-{\gamma_{y,\widehat{y,x}}}
\\
& {c_{y}\Cdot c_{yx}}
}
\]
of the rewriting system~$\Chin^{c}_2(n)$. We consider the Tietze transformation that substitutes the\linebreak rule~$\gamma_{y,yx}: c_y\Cdot c_{yx}\fl c_{yx}\Cdot c_y$ for the rule~$\epsilon_{x,y}$, for every~$1\leq x<y\leq n$.
Similarly, we substitute the rules~$\gamma_{yx,x}$,~$\gamma_{yy,x}$,~$\gamma_{y,xx}$,~$\gamma_{zy,x}$,~$\gamma_{zx,y}$ and~$\gamma_{z,yx}$ respectively for the rules~$\eta_{x,y}$,~$\epsilon_{x,y}$,~$\eta_{x,y}$,~$\epsilon_{x,y,z}$,~$\eta_{x, y, z}$ and~$\epsilon_{x,y,z}$ using the following critical branchings of the rewriting system~$\Chin^{c}_2(n)$:
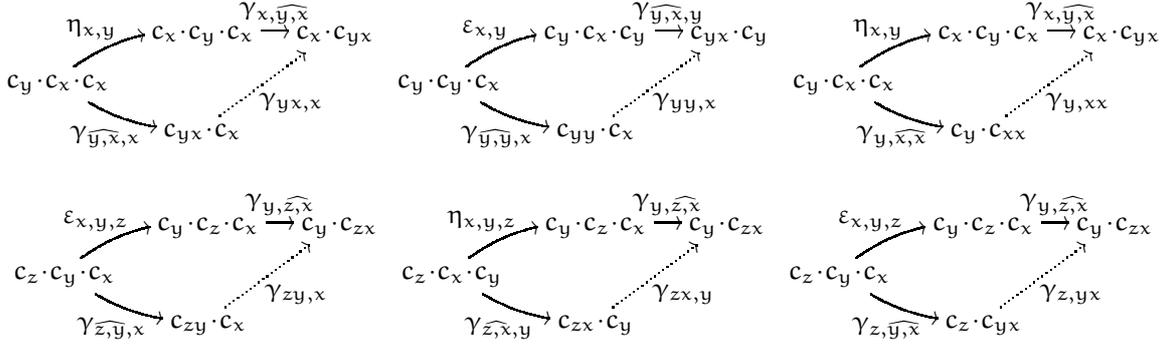
\[
\small
\xymatrix @C=1em @R=0.5em {
& {c_{x}\Cdot c_{y}\Cdot c_{x}}
  \ar[r] ^-{\gamma_{x,\widehat{y,x}}}
& {c_{x}\Cdot c_{yx}}
\\
{c_{y}\Cdot c_{x}\Cdot c_{x}}
	\ar@/^/ [ur] ^-{\eta_{x,y}}
	\ar@/_/ [dr] _-{\gamma_{\widehat{y,x},x}}
\\
& {c_{yx}\Cdot c_{x}}
\ar@{.>}[uur] _-{\gamma_{yx,x}}
}
\;
\xymatrix @C=1em @R=0.5em {
& {c_{y}\Cdot c_{x}\Cdot c_{y}}
  \ar[r] ^-{\gamma_{\widehat{y,x},y}}
& {c_{yx}\Cdot c_{y}}
\\
{c_{y}\Cdot c_{y}\Cdot c_{x}}
	\ar@/^/ [ur] ^-{\epsilon_{x, y}}
	\ar@/_/ [dr] _-{\gamma_{\widehat{y,y},x}}
\\
& {c_{yy}\Cdot c_{x}}
\ar@{.>}[uur] _-{\gamma_{yy,x}}
}
\;
\xymatrix @C=1em @R=0.5em {
& {c_{x}\Cdot c_{y}\Cdot c_{x}}
  \ar[r] ^-{\gamma_{x,\widehat{y,x}}}
& {c_{x}\Cdot c_{yx}}
\\
{c_{y}\Cdot c_{x}\Cdot c_{x}}
	\ar@/^/ [ur] ^-{\eta_{x,y}}
	\ar@/_/ [dr] _-{\gamma_{y,\widehat{x,x}}}
\\
& {c_{y}\Cdot c_{xx}}
\ar@{.>}[uur] _-{\gamma_{y,xx}}
}
\]
\[
\small
\xymatrix @C=1em @R=0.5em {
& {c_{y}\Cdot c_{z}\Cdot c_{x}}
  \ar[r] ^-{\gamma_{y,\widehat{z,x}}}
& {c_{y}\Cdot c_{zx}}
\\
{c_{z}\Cdot c_{y}\Cdot c_{x}}
	\ar@/^/ [ur] ^-{\epsilon_{x,y,z}}
	\ar@/_/ [dr] _-{\gamma_{\widehat{z,y},x}}
\\
& {c_{zy}\Cdot c_{x}}
\ar@{.>}[uur] _-{\gamma_{zy,x}}
}
\;
\xymatrix @C=1em @R=0.5em {
& {c_{y}\Cdot c_{z}\Cdot c_{x}}
  \ar[r] ^-{\gamma_{y,\widehat{z,x}}}
& {c_{y}\Cdot c_{zx}}
\\
{c_{z}\Cdot c_{x}\Cdot c_{y}}
	\ar@/^/ [ur] ^-{\eta_{x,y,z}}
	\ar@/_/ [dr] _-{\gamma_{\widehat{z,x},y}}
\\
& {c_{zx}\Cdot c_{y}}
\ar@{.>}[uur] _-{\gamma_{zx,y}}
}
\;
\xymatrix @C=1em @R=0.5em {
& {c_{y}\Cdot c_{z}\Cdot c_{x}}
  \ar[r] ^-{\gamma_{y,\widehat{z,x}}}
& {c_{y}\Cdot c_{zx}}
\\
{c_{z}\Cdot c_{y}\Cdot c_{x}}
	\ar@/^/ [ur] ^-{\epsilon_{x,y,z}}
	\ar@/_/ [dr] _-{\gamma_{z,\widehat{y,x}}}
\\
& {c_{z}\Cdot c_{yx}}
\ar@{.>}[uur] _-{\gamma_{z,yx}}
}
\]
The set of rules $\gamma_{-,-}$ obtained in this way is equal to $\Delta_2(n)$. This proves that the rewriting systems~$\Chin^{c}_2(n)$ and~$\PreColo_2(n)$ are Tietze equivalent.
\end{proof}

\subsection{Completion of the precolumn presentation}
\label{SSS:CompletionPrecolumnpresentation}
The rewriting system~$\PreColo_2(n)$ is not confluent, it has the following non-confluent critical branchings, that can be completed by Knuth-Bendix completion,~\cite{KnuthBendix70}, by the dotted arrows as follows:
\begin{center}
\begin{tabular}{@{} ccc @{}}
\begin{minipage}[t]{9cm}
{\bf i)} for every $1\leq x\leq y< z<t\leq n$ :
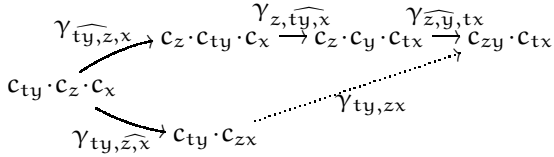
\[
\small
\xymatrix @C=1em @R=0.5em {
& {c_{z}\Cdot c_{ty}\Cdot c_{x}}
  \ar[r] ^-{\gamma_{z,\widehat{ty,x}}}
& {c_{z}\Cdot c_{y}\Cdot c_{tx}}
\ar[r] ^-{\gamma_{\widehat{z,y},tx}}
& {c_{zy}\Cdot c_{tx}}
\\
{c_{ty}\Cdot c_{z}\Cdot c_{x}}
	\ar@/^/ [ur] ^-{\gamma_{\widehat{ty,z},x}}
	\ar@/_/ [dr] _-{\gamma_{ty,\widehat{z,x}}}
\\
& {c_{ty}\Cdot c_{zx}}
\ar@{.>}[uurr] _-{\gamma_{ty,zx}}
}
\]
\end{minipage}
& 
\begin{minipage}[t]{7cm}
{\bf ii)} for every $1\leq x< y < z\leq n$ :
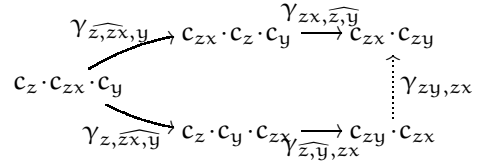
\[
\small 
\xymatrix @C=1.5em @R=0.5em {
& {c_{zx}\Cdot c_{z}\Cdot c_{y}}
  \ar[r] ^-{\gamma_{zx,\widehat{z,y}}}
& {c_{zx}\Cdot c_{zy}}
\\
{c_{z}\Cdot c_{zx}\Cdot c_{y}}
	\ar@/^/ [ur] ^-{\gamma_{\widehat{z,zx},y}}
	\ar@/_/ [dr] _-{\gamma_{z,\widehat{zx,y}}}
\\
& {c_{z}\Cdot c_{y}\Cdot c_{zx}}
\ar[r] _-{\gamma_{\widehat{z,y},zx}}
& {c_{zy}\Cdot c_{zx}}
\ar@{.>}[uu] _-{\gamma_{zy,zx}}
}
\]
\end{minipage}
\end{tabular}
\end{center}
\begin{center}
\begin{tabular}{@{} ccc @{}}
\begin{minipage}[t]{9cm}
{\bf iii)} for every $1\leq x<y\leq z<t\leq n$ : 
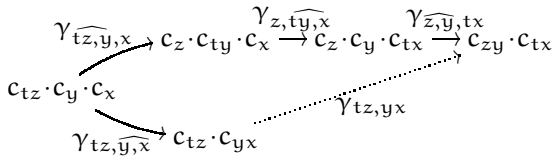
\[
\small
\xymatrix @C=1em @R=0.5em {
& {c_{z}\Cdot c_{ty}\Cdot c_{x}}
  \ar[r] ^-{\gamma_{z,\widehat{ty,x}}}
& {c_{z}\Cdot c_{y}\Cdot c_{tx}}
\ar[r] ^-{\gamma_{\widehat{z,y},tx}}
& {c_{zy}\Cdot c_{tx}}
\\
{c_{tz}\Cdot c_{y}\Cdot c_{x}}
	\ar@/^/ [ur] ^-{\gamma_{\widehat{tz,y},x}}
	\ar@/_/ [dr] _-{\gamma_{tz,\widehat{y,x}}}
\\
& {c_{tz}\Cdot c_{yx}}
\ar@{.>}[uurr] _-{\gamma_{tz,yx}}
}
\]
\end{minipage}
& 
\begin{minipage}[t]{7cm}
\quad 
\[
\small
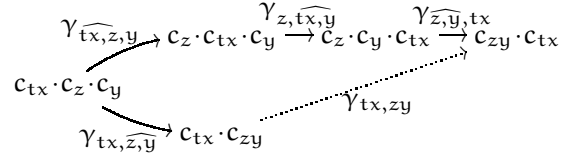
\xymatrix @C=1em @R=0.5em {
& {c_{z}\Cdot c_{tx}\Cdot c_{y}}
  \ar[r] ^-{\gamma_{z,\widehat{tx,y}}}
& {c_{z}\Cdot c_{y}\Cdot c_{tx}}
\ar[r] ^-{\gamma_{\widehat{z,y},tx}}
& {c_{zy}\Cdot c_{tx}}
\\
{c_{tx}\Cdot c_{z}\Cdot c_{y}}
	\ar@/^/ [ur] ^-{\gamma_{\widehat{tx,z},y}}
	\ar@/_/ [dr] _-{\gamma_{tx,\widehat{z,y}}}
\\
& {c_{tx}\Cdot c_{zy}}
\ar@{.>}[uurr] _-{\gamma_{tx,zy}}
}
\]
\end{minipage}
\end{tabular}
\end{center}
\begin{center}
\begin{tabular}{@{} ccc @{}}
\begin{minipage}[t]{9cm}
{\bf iv)} for every $1\leq x<y\leq z\leq n$ :
\[
\small
\xymatrix @C=2em @R=0.5em {
& {c_{zz}\Cdot c_{yx}}
\ar@{.>}[ddrr] ^-{\gamma_{zz,yx}}
\\
{c_{z}\Cdot c_{z}\Cdot c_{yx}}
	\ar@/^/ [ur] ^-{\gamma_{\widehat{z,z},yx}}
	\ar@/_/ [dr] _-{\gamma_{z,\widehat{z,yx}}}
\\
& {c_{z}\Cdot c_{y}\Cdot c_{zx}}
  \ar[r] _-{\gamma_{\widehat{z,y},zx}}
&{c_{zy}\Cdot c_{zx}}
 \ar[r] _-{\gamma_{zy,zx}}
  &{c_{zx}\Cdot c_{zy}}
}
\]
\end{minipage}
& 
\begin{minipage}[t]{7cm}
{\bf v)} for every $1 <x<y< n$ :
\[
\small
\xymatrix @C=2em @R=0.5em {
& {c_{yy}\Cdot c_{xx}}
\ar@{.>}[ddr] ^-{\gamma_{yy,xx}}
\\
{c_{y}\Cdot c_{y}\Cdot c_{xx}}
	\ar@/^/ [ur] ^-{\gamma_{\widehat{y,y},xx}}
	\ar@/_/ [dr] _-{\gamma_{y,\widehat{y,xx}}}
\\
& {c_{y}\Cdot c_{x}\Cdot c_{yx}}
  \ar[r] _-{\gamma_{\widehat{y,x},yx}}
&{c_{yx}\Cdot c_{yx}}
}
\]
\end{minipage}
\end{tabular}
\end{center}
\begin{center}
\begin{tabular}{@{} ccc @{}}
\begin{minipage}[t]{9cm}
{\bf vi)} for every $1\leq x\leq y< z\leq n$ :
\[
\small
\xymatrix @C=2em @R=0.5em {
& {c_{y}\Cdot c_{zx}\Cdot c_{x}}
\ar[r] ^-{\gamma_{y,\widehat{zx,x}}}
&{c_{y}\Cdot c_{x}\Cdot c_{zx}}
\ar[r] ^-{\gamma_{\widehat{y,x},zx}}
&{c_{yx}\Cdot c_{zx}}
\\
{c_{zy}\Cdot c_{x}\Cdot c_{x}}
	\ar@/^/ [ur] ^-{\gamma_{\widehat{zy,x},x}}
	\ar@/_/ [dr] _-{\gamma_{zy,\widehat{x,x}}}
\\ 
&{c_{zy}\Cdot c_{xx}}
  \ar@{.>}[uurr] _-{\gamma_{zy,xx}}
}
\]
\end{minipage}
& 
\begin{minipage}[t]{7cm}
{\bf vii)} for every $1<y<n$ :
\[
\small
\xymatrix @C=1em @R=0.5em {
& {c_{yy}\Cdot c_{y}}
\ar@{.>}[dd] ^-{\gamma_{yy,y}}
\\
{c_{y}\Cdot c_{y}\Cdot c_{y}}
	\ar@/^/ [ur] ^-{\gamma_{\widehat{y,y},y}}
	\ar@/_/ [dr] _-{\gamma_{y,\widehat{y,y}}}
\\
& {c_{y}\Cdot c_{yy}}
}
\]
\end{minipage}
\end{tabular}
\end{center}

The rules of~$\PreColo_2(n)$ together with the family of the dotted rules defined by {\bf i)}-{\bf vii)} form the set
\[
\big\{\; \gamma_{u,v} : c_u\Cdot c_v \fl R_{Q_n}(c_u\star_{\Ich} c_v) \;|\; c_u,c_v\in Q_{n}\;\big\}.
\]
That is, the set of rules of~$\Srs(Q_n, \C_n)$.
Finally, by this construction, we prove  that~$R_{Q_n}(c_u\star_{\Ich} c_v)$ is at most of length~$2$ in~$Q_n^{\ast}$, showing the semi-quadraticity of the presentation.

\begin{proposition}
\label{Proposition:ConvergenceColumnPresentation}
The rewriting system~$\Srs(Q_n,\C_n)$  is convergent.
\end{proposition}

\begin{proof}
Prove that $\Srs(Q_n,\C_n)$ is terminating. Consider the total order $\ordrecolChinese$ defined on $Q_n$ by 
\[
c_x\ordrecolChinese c_y \quad\text{ if }  x\leq y,
\qquad 
c_x \ordrecolChinese c_{zy} \quad \text{ if } x\leq y\leq z,
\]
\[
c_{yx}\ordrecolChinese c_z  \quad\text{ if } x<y\leq z,
\qquad
c_{yx} \ordrecolChinese  c_{tz} \quad\text{ if }  yx\leq_{\text{lex}}tz,
\]
where~$\leq_{\text{lex}}$ denotes the lexicographic order on $[n]^\ast$ induced by the natural order on~$[n]$.
Consider the map~$f : Q_n^\ast \fl  (\mathbb{N}, \leq)$  sending a word in~$ Q_n^\ast$ to its number of columns.
Define the length-lexicographic order~$\prec$ on~$Q_n^\ast$ with respect to~$\ordrecolChinese$ by setting, for all $u$ and $v$ in~$Q_n^\ast$:
\[
u\prec v \;\;\text{ if and only if}\;\;  \big(f(u)<f(v)\big) \; \text{ or } \; \big(f(u)=f(v) \text{ and } u  \ordrecolChinese^{\text{lex}} v \big),
\]
where~$\ordrecolChinese^{\text{lex}}$ denotes the lexicographic order on~$Q_n^\ast$ induced by~$\ordrecolChinese$. Any reduction with respect to~$\Srs(Q_n,\C_n)$  decrease a word in~$Q_n^\ast$ either with respect to~$f$ or with respect to~$\ordrecolChinese^{\text{lex}}$, showing that the rewriting system~$\Srs(Q_n, \C_n)$ is terminating.

Prove that $\Srs(Q_n,\C_n)$ is confluent. Any critical pair of~$\Srs(Q_n,\C_n)$ has the form~$(\gamma_{c_u,c_v}\Cdot c_t, c_u\Cdot \gamma_{c_v,c_t})$, for~$c_u, c_v, c_t$ in~$Q_n$. Note that, by associativity of~$\star_{\Ich}$, the rewriting path~$\sigma^{\vdash}_{R_{Q_n}(t)\cdot c_u}$ (resp.~$\sigma^{\vdash}_{c_u\cdot R_{Q_n}(t)}$ ) reduces~$R_{Q_n}(t)\cdot c_u$ (resp.~$ c_u\cdot R_{Q_n}(t)$) to~$R_{Q_n}(t\star_{\Ich} c_u)$ (resp.~$R_{Q_n}(c_u\star_{\Ich} t)$), for all~$t$ in~$\ChinT_n$ and~$c_u$ in~$Q_n$. 
Hence,  every critical pair of~$\Srs(Q_n,\C_n)$ has the following reduction diagram:
\[
\xymatrix @C=3.5em @R=0.5em {
&  R_{Q_n}(c_u\star_{\Ich} c_v ) \Cdot c_t
	\ar [rr] ^-{\sigma^{\vdash}_{ R_{Q_n}(c_u\star_{\Ich} c_v )\cdot c_t}}
&& R_{Q_n}(( c_u\star_{\Ich} c_v)\star_{\Ich} c_t )
\\
c_u \Cdot c_v \Cdot  c_t
	\ar@/^2ex/ [ur]^-{\gamma_{\widehat{u,v},t}}
	\ar@/_2ex/[dr]_-{\gamma_{u,\widehat{v,t}}}
\\
& c_u \Cdot  R_{Q_n}(c_v\star_{\Ich} c_t)
\ar [rr] _-{\sigma^{\vdash}_{c_u \cdot  R_{Q_n}(c_v\star_{\Ich} c_t)}}
&& R_{Q_n}(c_u \star_{\Ich} (c_v\star_{\Ich} c_t) )
}
\]
which is confluent by the associativity of the product~$\star_{\Ich}$. This proves that the rewriting system~$\Srs(Q_n,\C_n)$ is locally confluent and thus confluent by termination hypothesis. 
\end{proof}

\section{Chinese syzygies by insertions}
\label{S:ChineseSyzygies}

In this section we extend the rewriting system $\Srs(Q_n,\C_n)$ into a finite coherent convergent presentation of the Chinese monoid $\C_n$ with an explicit description of the generating syzygies.
By semi-quadraticity of~$\Srs(Q_n, \C_n)$, every rewriting path with source $c_u\Cdot c_v\Cdot c_t$ is an alternated composition of reductions of the form~(\ref{E:ReductionsLeftRight}). Moreover, every rewriting rule $\gamma_{-,-}$ of $\Srs(Q_n, \C_n)$ can be written 
\begin{equation}
\label{Equation:NewAlpha}
\gamma_{yx_1,x_2x_3} : c_{yx_1}\Cdot c_{x_2x_3} \fl c_{x_{\sigma(1)}x_{\sigma(2)}}\Cdot c_{yx_{\sigma(3)}}
\end{equation}
where $y\in [n]$, $x_1,x_2,x_3\in [n] \cup \{0\}$, $\sigma$ is a permutation on $[n] \cup \{0\}$, and $c_{x0}$ denotes the column generator~$c_{x}$ for any~$1<x<n$. 

\subsection{Remark}
\label{Remark:cyx1}
Note that when~$c_{yx_1}$ is not a square generator, then~$x_{\sigma(1)}$ takes value $y$ only if rule~(\ref{Equation:NewAlpha}) is one of the \emph{commutation rules} of the form
\begin{equation}
\label{commutationRelations}
 c_y\Cdot c_{yx}\fl c_{yx}\Cdot c_{y}, \quad 
c_{zy}\Cdot c_{zx}\fl c_{zx}\Cdot c_{zy},\quad
c_{yy}\Cdot c_{y} \fl c_{y}\Cdot c_{yy},\quad
c_{yy}\Cdot c_{yx}\fl c_{yx}\Cdot c_{yy}
\end{equation}
for~$x<y<z$. When~$c_{yx_1}$ is a square generator, with~$y>x_2$, then~$x_{\sigma(1)}$ takes value $y$ only if rule~(\ref{Equation:NewAlpha}) is one of the form
\begin{equation}
\label{SquareRelations}
 c_{yy}\Cdot c_{x}\fl c_{yx}\Cdot c_{y},\quad
 c_{yy}\Cdot c_{xx}\fl c_{yx}\Cdot c_{yx},\quad
  c_{zz}\Cdot c_{yx} \fl c_{zx}\Cdot c_{zy}.
\end{equation}

We obtain the following bounds for the rewriting paths with source a critical branching of~$\Srs(Q_n, \C_n)$.

\begin{proposition}
\label{P:InequalitiesLengthReductions}
For all~$c_u,c_v,c_t$ in~$Q_n$ such that~$c_u\Cdot c_v$ and~$c_v\Cdot c_t$ are not Chinese normal forms, the two following inequalities hold:
\begin{equation}
\label{E:InequalityLength}
\len_l(c_u\Cdot c_v\Cdot c_t)\leq 5,
\qquad
\text{and}
\qquad
\len_r(c_u\Cdot c_v\Cdot c_t)\leq 5.
\end{equation}
\end{proposition}

The proof of this result is based on the two following lemmata~\ref{CommutationLemma} and \ref{SquareLemma}.
Let~$c_{u},c_{v},c_{t}$ be in~$Q_n$ such that~$c_u\Cdot c_v$ and~$c_v\Cdot c_t$ are not Chinese normal forms. The Chinese normal form of the word~$c_u\Cdot c_v\Cdot c_t$  can be obtained by applying one, two or three steps of reductions of the leftmost normalization  strategy of~$\Srs(Q_n, \C_n)$. In this case, we have~$\len_l(c_u\Cdot c_v\Cdot c_t)\leq 3$. Otherwise, the following lemma shows that \linebreak $\len_l(c_u\Cdot c_v\Cdot c_t)\leq 5$.

\begin{lemma}
\label{CommutationLemma}
Let~$c_{u},c_{v},c_{t}$ be in~$Q_n$ such that~$c_u\Cdot c_v$ and~$c_v\Cdot c_t$ are not Chinese normal forms.
Suppose that the word obtained after three steps of reductions of the leftmost normalization strategy of~$\Srs(Q_n, \C_n)$ with source~$c_u\cdot c_v\cdot c_t$ is not a Chinese normal form. Then, the Chinese normal form of this word is obtained by applying at most two steps of reductions, that consist only on the commutation rules~(\ref{commutationRelations}).
\end{lemma}

\begin{proof}
Let~$c_{yx_1}$, $c_{x_2x_3},c_{x_4x_5}$ be in~$Q_n$ such that~$c_{yx_1}\Cdot c_{x_2x_3}$ and~$c_{x_2x_3}\Cdot c_{x_4x_5}$ are not Chinese normal forms. By definition of~$\Srs(Q_n, \C_n)$, we have
\begin{equation}
\label{E:reductionCommutation}
\begin{array}{r}
c_{yx_1}\Cdot c_{x_2x_3}\Cdot c_{x_4x_5} 
\fl c_{x_{\sigma(1)}x_{\sigma(2)}}\Cdot c_{yx_{\sigma(3)}}\Cdot c_{x_4x_5} 
\fl c_{x_{\sigma(1)}x_{\sigma(2)}}\Cdot c_{x_{\sigma'(\sigma(3))}x_{\sigma'(4)}}\Cdot c_{yx_{\sigma'(5)}} \hspace{1.5cm}\\
\fl 
c_{z_1 z_2}\Cdot c_{x_{\sigma(1)}z_3}\Cdot c_{yx_{\sigma'(5)}}
\end{array}
\end{equation}
with~$z_1= x_{\sigma''(\sigma(2))}$,~$z_2= x_{\sigma''(\sigma'(\sigma(
3)))}$,~$z_3= x_{\sigma''(\sigma'(4))}$, and where~$\sigma$, $\sigma'$,~$\sigma''$ are permutations on~$[n]\cup \{0\}$, and~$c_{x_{\sigma(1)}x_{\sigma(2)}}\Cdot c_{yx_{\sigma(3)}}$, ~$c_{x_{\sigma'(\sigma(3))}x_{\sigma'(4)}}\Cdot c_{yx_{\sigma'(5)}}$,~$c_{z_1 z_2}\Cdot c_{x_{\sigma(1)}z_3}$ are Chinese normal forms.

Suppose that~$c_{x_{\sigma(1)}z_3}\Cdot  c_{yx_{\sigma'(5)}}$ is not a Chinese normal form. Following Remark~\ref{Remark:cyx1}, its only possible reductions are of form~(\ref{commutationRelations}) or~(\ref{SquareRelations}). Let us prove that the rules ~(\ref{SquareRelations}) cannot be applied. Suppose the contrary. Then~$x_{\sigma(1)} = z_3 > y$. Since~$c_{z_1 z_2}\Cdot c_{x_{\sigma(1)}z_3}$ is a Chinese normal form, we obtain that~$z_1 = z_3$ and~\mbox{$c_{x_{\sigma(1)}x_{\sigma(2)}}\Cdot c_{x_{\sigma'(\sigma(3))}x_{\sigma'(4)}}\Cdot c_{yx_{\sigma'(5)}} = c_{z_3z_3}\Cdot c_{z_3 z_2}\Cdot c_{yx_{\sigma'(5)}}$}. 
Since~$z_3>y$, this proves that~$c_{z_3 z_2}\Cdot c_{yx_{\sigma'(5)}} = c_{x_{\sigma'(\sigma(3))}x_{\sigma'(4)}}\Cdot c_{yx_{\sigma'(5)}}$ is not a Chinese normal form, which yields a contradiction.

Then we can only apply a commutation rule on~$c_{x_{\sigma(1)}z_3} \Cdot c_{yx_{\sigma'(5)}}$, with~$x_{\sigma(1)}= y$, and we rewrite the word~$c_{z_1 z_2}\Cdot c_{x_{\sigma(1)}z_3}\Cdot c_{yx_{\sigma'(5)}}$ into~$c_{z_1 z_2}\Cdot c_{yx_{\sigma'(5)}} \Cdot c_{x_{\sigma(1)}z_3}$. Suppose that~$c_{z_1 z_2}\Cdot c_{yx_{\sigma'(5)}}$ is not a Chinese normal form, then we can apply on it a rule of type~(\ref{commutationRelations}) or~(\ref{SquareRelations}). 
As in the previous step, let us prove that the rules ~(\ref{SquareRelations}) cannot be applied. Suppose the contrary.  Then~$z_1=z_2>y$. Since~$c_{z_1 z_2}\Cdot c_{x_{\sigma(1)}z_3}$ is a Chinese normal form, we obtain that~$z_1=z_2= x_{\sigma(1)}= y$, which yields a contradiction.
Then we can only apply a commutation rule on~$c_{z_1 z_2}\Cdot c_{yx_{\sigma'(5)}}$.

We have thus proved that the Chinese normal form of the word~$c_{yx_1}\Cdot c_{x_2x_3}\Cdot c_{x_4x_5}$ is obtained by applying at most two steps of reductions that consist only on the commutation rules.
\end{proof}

\begin{lemma}
\label{SquareLemma}
For all~$c_u,c_v,c_t$ in~$Q_n$ such that~$c_u$ is a square generator and  the words~$c_{u}\Cdot c_v$ and~$c_v\Cdot c_t$ are not Chinese normal forms, the inequality~$\len_r(c_{u}\Cdot c_v\Cdot c_t)\leq 5$ holds.
\end{lemma}

\begin{proof}
By hypotheses, the word~$c_{u}\Cdot c_v\Cdot c_t$ has the following forms: $c_{rr}\Cdot c_{tz}\Cdot c_{yx}$ and $c_{rr}\Cdot c_{tx}\Cdot c_{zy}$, for all~$x<y<z<t\leq r$, $c_{tt}\Cdot c_{zy}\Cdot c_{zx}$, $c_{tt}\Cdot c_{zx}\Cdot c_{y}$, $c_{tt}\Cdot c_{zy}\Cdot c_{x}$ and $c_{tt}\Cdot c_{zy}\Cdot c_{yx}$, for all~$x<y<z\leq t$,  $c_{zz}\Cdot c_{yx}\Cdot c_{x}$ and $c_{zz}\Cdot c_{y}\Cdot c_{x}$, for all~$x<y\leq z$, $c_{rr}\Cdot c_{ty}\Cdot c_{zx}$, for all~$x\leq y<z<t\leq r$, and $c_{tt}\Cdot c_{z}\Cdot c_{yx}$ for all~$x<y\leq z \leq t$.
For all these forms, one can check that~$\len_r(c_{u}\Cdot c_v\Cdot c_t)\leq 5$.
\end{proof}

\subsection{Proof of Proposition~\ref{P:InequalitiesLengthReductions}}
Let~$c_{yx_1}$, $c_{x_2x_3}$,~$c_{x_4x_5}$ be in~$Q_n$ such that~$c_{yx_1}\Cdot c_{x_2x_3}$ and~$c_{x_2x_3}\Cdot c_{x_4x_5}$ are not Chinese  normal forms. Let us prove that~$\len_l(c_{yx_1}\Cdot c_{x_2x_3}\Cdot c_{x_4x_5} )\leq 5$. 
Suppose that the word obtained after two steps of reductions of  the leftmost normalization strategy of~$\Srs(Q_n, \C_n)$ with source~$c_{yx_1}\Cdot c_{x_2x_3}\Cdot c_{x_4x_5}$ is not a Chinese normal form.
Consider a reduction as in~(\ref{E:reductionCommutation}), and suppose that~$c_{x_{\sigma(1)}z_3}\Cdot c_{yx_{\sigma'(5)}}$ is not a Chinese normal form.
By Lemma~\ref{CommutationLemma} its only possible reductions are commutation rules, hence there is a reduction~$c_{z_1 z_2}\Cdot c_{x_{\sigma(1)}z_3}\Cdot c_{yx_{\sigma'(5)}} \fl c_{z_1 z_2}\Cdot c_{yx_{\sigma'(5)}}\Cdot c_{x_{\sigma(1)}z_3}$.
Suppose that~$c_{z_1 z_2}\Cdot c_{yx_{\sigma'(5)}}$ is not a Chinese normal form, then by the same argument there is a reduction~$c_{z_1 z_2}\Cdot c_{yx_{\sigma'(5)}} \Cdot c_{x_{\sigma(1)}z_3}\fl c_{yx_{\sigma'(5)}}\Cdot c_{z_1 z_2}\Cdot c_{x_{\sigma(1)}z_3}$, 
where~$c_{yx_{\sigma'(5)}}\Cdot c_{x_{\sigma(1)}z_3}$ and~\mbox{$c_{yx_{\sigma'(5)}}\Cdot c_{z_1 z_2}$} are Chinese normal forms. Since~$c_{z_1 z_2}c_{x_{\sigma(1)}z_3}$ is a Chinese normal form, we obtain that~$c_{yx_{\sigma'(5)}}c_{x_{\sigma(1)}z_3}$ is a Chinese normal form.
This proves the first inequality in~\eqref{E:InequalityLength}.

Let us prove that~$\len_r(c_{yx_1}\Cdot c_{x_2x_3}\Cdot c_{x_4x_5} )\leq 5$. 
Suppose that the word obtained after three steps of reductions of the rightmost normalization strategy of~$\Srs(Q_n, \C_n)$ with source~$c_{yx_1}\Cdot c_{x_2x_3}\Cdot c_{x_4x_5}$ is not a Chinese normal form. 
By definition of~$\Srs(Q_n, \C_n)$, we have the following reductions
\begin{equation}
  \label{Eq:ChineseRightRed}
  \begin{aligned}
  & & c_{yx_1}\Cdot c_{x_2x_3}\Cdot c_{x_4x_5}
\fl
 c_{yx_1}\Cdot c_{x_{\sigma(3)}x_{\sigma(4)}}\Cdot c_{x_2x_{\sigma(5)}} 
\fl
c_{x_{\sigma'(1)}y_1}\Cdot c_{yy_2}\Cdot c_{x_2x_{\sigma(5)}} \\    
   &&c_{x_{\sigma'(1)}y_1}\Cdot c_{yy_2}\Cdot c_{x_2x_{\sigma(5)}} 
\fl 
c_{x_{\sigma'(1)}y_1}\Cdot  c_{x_{\sigma"(2)}z_1}\Cdot c_{yz_2} 
\fl 
c_{t_1t_2}\Cdot c_{x_{\sigma'(1)}t_3}\Cdot c_{yz_2} 
  \end{aligned}
\end{equation}
with~$y_1 =x_{\sigma'(\sigma(3))}$,~$y_2= x_{\sigma'(\sigma(4))}$,~$z_1=x_{\sigma''(\sigma'(\sigma(4)))}$,~$z_2= x_{\sigma''(\sigma(5))}$,~$t_1= x_{\sigma_1(\sigma'(1))}$,~$t_2= x_{\sigma_1(\sigma'(\sigma(3)))}$, $t_3= x_{\sigma_1(\sigma''(\sigma'(\sigma(1))))}$, and where~$\sigma,\sigma',\sigma'',\sigma_1$ are permutations on~$[n]\cup \{0\}$, and~\mbox{$c_{x_{\sigma(3)}x_{\sigma(4)}}\Cdot c_{x_2x_{\sigma(5)}}$,} $c_{x_{\sigma'(1)}y_1}\Cdot c_{yy_2}$,~$c_{x_{\sigma"(2)}z_1}\Cdot c_{yz_2}$ and~$c_{t_1t_2}\Cdot c_{x_{\sigma'(1)}t_3}$ are Chinese normal forms.

Suppose that the word obtained after applying four steps of reductions of the rightmost normalization strategy with source~$c_{yx_1}\Cdot c_{x_2x_3}\Cdot c_{x_4x_5}$ is not a Chinese normal form.
Then~$x_{\sigma'(1)} = y$ and the second reduction of~(\ref{Eq:ChineseRightRed})  is~$c_{yx_1}\Cdot c_{x_{\sigma(3)}x_{\sigma(4)}}\Cdot c_{x_2x_{\sigma(5)}} \fl c_{yy_1}\Cdot c_{yy_2}\Cdot c_{x_2x_{\sigma(5)}}$. Following Remark~\ref{Remark:cyx1}, the rule~$\gamma_{yx_1, x_{\sigma(3)}x_{\sigma(4)}}$ is of form~(\ref{commutationRelations}) or~(\ref{SquareRelations}). Let us prove that it cannot be of form~(\ref{commutationRelations}). Suppose the contrary. Since~$c_{x_{\sigma(3)}x_{\sigma(4)}}\Cdot c_{x_2x_{\sigma(5)}}$ is a Chinese  normal form, we obtain~$x_{\sigma(3)}=y \geq x_2$. Moreover, since~$c_{yx_1}\Cdot c_{x_2x_3}$ is not a Chinese normal form, the inequality~$y\leq x_2$ holds, hence~$y=x_2$. In this way, the first reduction of~(\ref{Eq:ChineseRightRed}) is~\mbox{$c_{yx_1}\Cdot c_{yx_3}\Cdot c_{yx_5} 
\fl c_{{y}x_{3}}\Cdot c_{yx_{1}}\Cdot c_{yx_5}$}, where~$c_{yx_3}c_{yx_5}$ is a Chinese normal form, and  its second reduction is~$c_{{y}x_{3}}\Cdot c_{yx_{1}}\Cdot c_{yx_5}\fl c_{{y}x_{3}}\Cdot c_{yx_5}\Cdot c_{yx_{1}}$. Since the word obtained after three steps of reductions of the rightmost normalization strategy of~$\Srs(Q_n, \C_n)$ with source~$c_{yx_1}\Cdot c_{yx_3}\Cdot c_{yx_5}$ is not a Chinese normal form,  the word~$c_{yx_{3}}\Cdot c_{yx_5}$ is not a Chinese normal form, which yields a contradiction. 

Thus, the rule~$\gamma_{yx_1, x_{\sigma(3)}x_{\sigma(4)}}$  is of form~(\ref{SquareRelations}) and~$c_{yx_1}$ is a square generator such that~$c_{yx_1}\Cdot c_{x_2x_3}$ and~$c_{x_2x_3}\Cdot c_{x_4x_5}$ are not Chinese normal forms. Hence by Lemma~\ref{SquareLemma} we obtain~$\len_r(c_{yx_1}\Cdot c_{x_2x_3}\Cdot c_{x_4x_5})\leq 5$. This proves the second inequality in~\eqref{E:InequalityLength}.

\begin{theorem}
\label{T:CoherenceQCol3}
The rewriting system $\Srs(Q_n, \C_n)$ extends into a finite coherent convergent presentation of the Chinese monoid $\C_n$ by adjunction of a generating syzygy
\[
\xymatrix @C=2.8em @R=0.65em {
&
c_{e}\Cdot  c_{e'}\Cdot c_t
   \ar@2[r] ^-{\gamma_{e,\widehat{e',t}}} 
 \ar@3 []!<95pt,-10pt>;[dd]!<95pt,10pt> ^{\;\mathcal{X}_{u,v,t}} 
&
c_{e}\Cdot c_{b}\Cdot c_{b'}
   \ar@2[r] ^-{\gamma_{\widehat{e,b},b'}} 
&
c_{s}\Cdot c_{s'}\Cdot c_{b'}
   \ar@2[r] ^-{\gamma_{s,\widehat{s',b'}}}
&
{c_s\Cdot c_k\Cdot  c_{k'}}
 \ar@2@/^/[dr] ^-{\gamma_{\widehat{s,k},k'}}
\\
c_u\Cdot c_v\Cdot c_t
   \ar@2@/^/[ur] ^-{\gamma_{\widehat{u,v},t}}
   \ar@2@/_/[dr] _-{\gamma_{u,\widehat{v,t}}}
&&&&&
c_{\Jch}\Cdot c_{m}\Cdot c_{k'}
\\
&
c_u\Cdot c_{w}\Cdot c_{w'}
   \ar@2[r] _-{\gamma_{\widehat{u,w},w'}}
&
c_{a}\Cdot c_{a'}\Cdot c_{w'}
  \ar@2[r] _-{\gamma_{a,\widehat{a',w'}}}
&
c_{a}\Cdot c_{d}\Cdot c_{d'}
  \ar@2[r] _-{\gamma_{a,\widehat{a',w'}}}
& c_{\Jch}\Cdot c_{l'}\Cdot c_{d'}
  \ar@2@/_/[ur] _-{\gamma_{l,\widehat{l',d'}}}
}	
\]
for all~$c_u,c_v,c_t$ in~$Q_n$ such that~$c_u\Cdot c_v$ and~$c_v\Cdot c_t$ are not Chinese  normal forms, and where the $2$-cells~$\gamma_{-,-}$ denote either a rewriting rule of $\Srs(Q_n, \C_n)$ or an identity.
\end{theorem}
\begin{proof}
Any critical branching of $\Srs(Q_n, \C_n)$ has the form
\[
\xymatrix @C=2.75em @R=0.4em {
& R_{Q_n}(c_u\star_{\Ich} c_v)\Cdot c_t
\\
c_u\Cdot c_v \Cdot c_t
	\ar@2@/^2ex/ [ur] ^-{\gamma_{\widehat{u,v},t}}
	\ar@2@/_2ex/ [dr] _-{\gamma_{u,\widehat{v,t}}} 
&
\\
& {c_u\Cdot R_{Q_n}(c_v\star_{\Ich} c_t)}
}
\]
for all $c_u ,c_v, c_t$ in $Q_n$ such that~$c_u\Cdot c_v$ and~$c_v\Cdot c_t$ are not Chinese  normal forms, that is confluent by Theorem~\ref{Theorem:FSQCPChinese}. Moreover by Proposition~\ref{P:InequalitiesLengthReductions}, $\len_l(c_u\Cdot c_v\Cdot c_t)\leq 5$ and~$\len_r(c_u\Cdot c_v\Cdot c_t)\leq 5$.
We conclude with  Squier's coherence theorem recalled in Subsection~\ref{SS:CoherentPresentation}.
\end{proof}

Note that some $2$-cells~$\gamma_{-,-}$ in the boundary of the generating syzygy $\Xr_{u,v,t}$ can be identity. However, following construction given in the proof of Proposition~\ref{P:InequalitiesLengthReductions}, if the source (resp. target) of~$\Xr_{u,v,t}$ is of length $5$, then its target (resp. source) is of length at most $4$.

\subsection{Relations among the insertion algorithms}
\label{SS:RelationsAmongInsertionsAlgorithms}
Note that the generating syzygies of the coherent presentation of the monoid~$\C_n$ obtained in  Theorem~\ref{T:CoherenceQCol3} can be interpreted in terms of the right and left insertion algorithms as follows.
Consider the rewriting system on~$Q_n$, whose rules are 
\[
c_u\Cdot c_v \fl R_{Q_n}(c_v\star_{\Jch}c_u),
\]
for all~$c_u,c_v$ in~$Q_n$ such that~$c_u\Cdot c_v \neq R_{Q_n}(c_v\star_{\Jch}c_u)$.
By Corollary~\ref{C:oppositeSDSAssociativitycondition}, the equality~$R_{Q_n} (c_v\star_{\Jch}c_u)= R_{Q_n}(c_u\star_{\Ich} c_v)$ holds for all~$c_u,c_v$ in~$Q_n$, and thus this rewriting system coincides with~$\Srs(Q_n,\C_n)$.  Hence, the generating syzygy of the coherent presentation of  Theorem~\ref{T:CoherenceQCol3} has the following form
\[
\xymatrix @C=2.7em @R=-0.1em {
c_u\Cdot c_v\Cdot c_t
	\ar@2@/^4ex/ [rr] ^{\sigma^{\vdash}_{c_u\cdot c_v\cdot c_t}} _{}="src"
	\ar@2@/_4ex/ [rr] _{\sigma^{\dashv}_{c_u\cdot c_v\cdot c_t}}^{}="tgt"
&& R_{Q_n}(c_u \star_{\Ich}c_v \star_{\Ich} c_t)
\ar@3 "src"!<0pt,-10pt>;"tgt"!<0pt,10pt> ^-{}
}
\]
for all~$c_u, c_v, c_t$ in $Q_n$ such that $c_u\Cdot c_v\neq R_{Q_n}(c_u\star_{\Ich} c_v)$ and $c_v\Cdot c_t \neq R_{Q_n}(c_v\star_{\Ich} c_t)$, where the application of the leftmost (resp. rightmost) normalization strategy~$\sigma^{\vdash}$ (resp.~$\sigma^{\dashv}$) on the word $c_u\Cdot c_v\Cdot c_t$ corresponds to the application of the right (resp. left) insertion 
\[
 R_{Q_n}\big(\emptyset \insr{\Ich} \R(c_u)\R(c_v)\R(c_t)\big)
\qquad
\text{\big(\,resp. }
 R_{Q_n}(\R(c_u)\R(c_v)\R(c_t) \insl{} \emptyset)
\text{\;\big).}
\]

\subsection{Actions of Chinese monoids on categories} 
\label{SSS:ActionsChineseMonoids}
A monoid~$\M$ can be seen as a $2$-category with exactly one $0$-cell~$\bullet$, with the elements of the monoid~$\M$ as $1$-cells and with identity $2$-cells only. The category of \emph{actions of $\M$ on categories} is the category~$\mathrm{Act}(\M)$ of $2$-representations of~$\M$ in the category~$\Cat$ of categories. The full subcategory of~~$\mathrm{Act}(\M)$ whose objects are the $2$-functors is denoted by~$2\Cat(\M,\Cat)$. We refer the reader to~\cite{GaussentGuiraudMalbos15} for a full introduction on the category of $2$-representations of $2$-categories.
More explicitly, an action~$A$ of the monoid~$\M$ is specified by a category~\mbox{$\C=A(\bullet)$}, an endofunctor~\mbox{$A(u):\C\to\C$} for every~$u$ in~$\M$, a natural isomorphism~\mbox{$A_{u,v}:A(u)A(v)\dfl A(uv)$} for every elements~$u$ and~$v$ of~$\M$, and a natural isomorphism~$A_{\bullet}:1_{\C}\dfl A(1)$ such that:
\begin{enumerate}[i)]
\item for every triple~$(u,v,w)$ of elements of the monoid~$\M$, the following diagram commutes
\[
\xymatrix @!C @R=0.4em @C=0.5em{
& A(uv)A(w)
	\ar@2 @/^2ex/ [dr] ^-{A_{uv,w}}
	\ar@{} [dd] |-{=}
\\
A(u)A(v)A(w)
	\ar@2 @/^2ex/ [ur] ^-{A_{u,v}A(w)}
	\ar@2 @/_2ex/ [dr] _-{A(u)A_{v,w}}
&& A(uvw)
\\
& A(u)A(vw)
	\ar@2 @/_2ex/ [ur] _-{A_{u,vw}}
}
\]
\item for every element~$u$ of the monoid~$\M$, the following diagrams commute
\[
\xymatrix @!C @R=0.4em @C=0.5em {
& A(1)A(u)
	\ar@2 @/^2ex/ [dr] ^-{A_{1,u}}
\\
A(u)
	\ar@2 @/^2ex/ [ur] ^-{A_{\bullet} A(u)}
	\ar@{=} [rr] _-{}="tgt"
&& A(u)
	\ar@{} "1,2";"tgt" |-{=}
}
\qquad\qquad
\xymatrix @!C @R=0.4em {
& A(u)A(1)
	\ar@2 @/^2ex/ [dr] ^-{A_{u,1}}
\\
A(u) 
	\ar@2 @/^2ex/ [ur] ^-{A(u) A_{\bullet}}
	\ar@{=} [rr] _-{}="tgt"
&& A(u)
	\ar@{} "1,2";"tgt" |-{=}
}
\]
\end{enumerate}

Let $\M$ be a monoid and let~$\Sigma$ be an extended presentation of~$\M$. The $(3,1)$-polygraph~$\Sigma$ is a coherent presentation of~$\M$ if, and only if, for every $2$-category~$\mathcal{C}$, there is an equivalence of categories between~$\mathrm{Act}(\M)$ and~$2\Cat(\Sigma_1^{\ast}/\Sigma_2,\mathcal{C})$,
that is natural in~$\mathcal{C}$,~\cite{GaussentGuiraudMalbos15}.
In this way, up to equivalence, the actions of a monoid~$\M$ on categories are the same as the $2$-functors from~$\Sigma_1^{\ast}/\Sigma_2$ to $\Cat$.

Using this description,  Theorem~\ref{T:CoherenceQCol3} allows us to present actions of Chinese monoids on categories as follows:

\begin{theorem}
\label{Theorem:ActionsOnCategories}
The category $\Act(\C_n)$ of actions of the Chinese monoid $\C_n$ on categories is equivalent to the category of  $2$-functors from the free $(2, 1)$-category~$\tck{\Srs(Q_n, \C_n)}$  generated by the rewriting system~$\Srs(Q_n, \C_n)$ to the category~$\Cat$ of categories, that sends any generating syzygy $\Xr_{u,v,t}$ to commutative diagrams in the category $\Cat$.
\end{theorem}

\begin{small}
\renewcommand{\refname}{\Large\textsc{References}}
\bibliographystyle{plain}
\bibliography{biblioPLAXIQUE}

\def\cprime{$'$}
\begin{thebibliography}{10}

\bibitem{BookOtto93}
Ronald Book and Friedrich Otto.
\newblock {\em String-rewriting systems}.
\newblock Texts and Monographs in Computer Science. Springer-Verlag, 1993.

\bibitem{CainGrayMalheiroChinese}
Alan~J. Cain, Robert~D. Gray, and Ant{\'o}nio Malheiro.
\newblock Rewriting systems and biautomatic structures for {C}hinese,
  hypoplactic, and {S}ylvester monoids.
\newblock {\em Internat. J. Algebra Comput.}, 25(1-2):51--80, 2015.

\bibitem{CainMalheiroSilva19}
Alan~J. Cain, António Malheiro, and Fábio~M. Silva.
\newblock The monoids of the patience sorting algorithm.
\newblock {\em International Journal of Algebra and Computation},
  29(01):85--125, 2019.

\bibitem{CassaigneEspieKrobNovelliHibert01}
Julien Cassaigne, Marc Espie, Daniel Krob, Jean-Christophe Novelli, and Florent
  Hivert.
\newblock The {C}hinese monoid.
\newblock {\em Internat. J. Algebra Comput.}, 11(3):301--334, 2001.

\bibitem{ChenQiu08}
Yuqun Chen and Jianjun Qiu.
\newblock Gr\"obner-{S}hirshov basis for the {C}hinese monoid.
\newblock {\em J. Algebra Appl.}, 7(5):623--628, 2008.

\bibitem{OConnell03}
Neil~O Connell.
\newblock Conditioned random walks and the {RSK} correspondence.
\newblock {\em Journal of Physics A: Mathematical and General},
  36(12):3049--3066, mar 2003.

\bibitem{DuchampKrob94}
G\'erard Duchamp and Daniel Krob.
\newblock Plactic-growth-like monoids.
\newblock In {\em Words, languages and combinatorics, {II} ({K}yoto, 1992)},
  pages 124--142. World Sci. Publ., River Edge, NJ, 1994.

\bibitem{Fulton97}
William Fulton.
\newblock {\em Young tableaux}, volume~35 of {\em London Mathematical Society
  Student Texts}.
\newblock Cambridge University Press, Cambridge, 1997.
\newblock With applications to representation theory and geometry.

\bibitem{GaussentGuiraudMalbos15}
St{\'e}phane Gaussent, Yves Guiraud, and Philippe Malbos.
\newblock Coherent presentations of {A}rtin monoids.
\newblock {\em Compos. Math.}, 151(5):957--998, 2015.

\bibitem{Giraudo12}
Samuele Giraudo.
\newblock Algebraic and combinatorial structures on pairs of twin binary trees.
\newblock {\em J. Algebra}, 360:115--157, 2012.

\bibitem{GuiraudMalbos12advances}
Yves Guiraud and Philippe Malbos.
\newblock Higher-dimensional normalisation strategies for acyclicity.
\newblock {\em Adv. Math.}, 231(3-4):2294--2351, 2012.

\bibitem{GuiraudMalbos10smf}
Yves Guiraud and Philippe Malbos.
\newblock Identities among relations for higher-dimensional rewriting systems.
\newblock volume~26 of {\em S\'{e}min. Congr.}, pages 145--161. Soc. Math.
  France, Paris, 2013.

\bibitem{GuiraudMalbos18}
Yves Guiraud and Philippe Malbos.
\newblock Polygraphs of finite derivation type.
\newblock {\em Math. Structures Comput. Sci.}, 28(2):155--201, 2018.

\bibitem{Karpuz10}
Eylem G\"uzel~Karpuz.
\newblock Complete rewriting system for the {C}hinese monoid.
\newblock {\em Appl. Math. Sci. (Ruse)}, 4(21-24):1081--1087, 2010.

\bibitem{HageMalbos17}
Nohra Hage and Philippe Malbos.
\newblock Knuth's coherent presentations of plactic monoids of type {A}.
\newblock {\em Algebras and Representation Theory}, 20(5):1259--1288, Oct 2017.

\bibitem{HivertNovelliThibon05}
F.~Hivert, J.-C. Novelli, and J.-Y. Thibon.
\newblock The algebra of binary search trees.
\newblock {\em Theoret. Comput. Sci.}, 339(1):129--165, 2005.

\bibitem{HivertNovelliThibon07}
Florent Hivert, Jean-Christophe Novelli, and Jean-Yves Thibon.
\newblock Commutative combinatorial hopf algebras.
\newblock {\em Journal of Algebraic Combinatorics}, 28(1):65, Jun 2007.

\bibitem{KnuthBendix70}
Donald Knuth and Peter Bendix.
\newblock Simple word problems in universal algebras.
\newblock In {\em Computational {P}roblems in {A}bstract {A}lgebra ({P}roc.
  {C}onf., {O}xford, 1967)}, pages 263--297. Pergamon, Oxford, 1970.

\bibitem{Knuth70}
Donald~E. Knuth.
\newblock Permutations, matrices, and generalized {Y}oung tableaux.
\newblock {\em Pacific J. Math.}, 34:709--727, 1970.

\bibitem{LascouxSchutsenberger81}
Alain Lascoux and Marcel-Paul Sch{\"u}tzenberger.
\newblock Le mono\"\i de plaxique.
\newblock In {\em Noncommutative structures in algebra and geometric
  combinatorics ({N}aples, 1978)}, volume 109 of {\em Quad. ``Ricerca Sci.''},
  pages 129--156. CNR, Rome, 1981.

\bibitem{Lecouvey02}
C{\'e}dric Lecouvey.
\newblock Schensted-type correspondence, plactic monoid, and jeu de taquin for
  type {$C_n$}.
\newblock {\em J. Algebra}, 247(2):295--331, 2002.

\bibitem{Lecouvey03}
C{\'e}dric Lecouvey.
\newblock Schensted-type correspondences and plactic monoids for types {$B_n$}
  and {$D_n$}.
\newblock {\em J. Algebraic Combin.}, 18(2):99--133, 2003.

\bibitem{Lothaire02}
M.~Lothaire.
\newblock {\em Algebraic combinatorics on words}, volume~90 of {\em
  Encyclopedia of Mathematics and its Applications}.
\newblock Cambridge University Press, Cambridge, 2002.

\bibitem{Novelli00}
Jean-Christophe {Novelli}.
\newblock {On the hypoplactic monoid.}
\newblock {\em {Discrete Math.}}, 217(1-3):315--336, 2000.

\bibitem{Priez2013}
Jean-Baptiste Priez.
\newblock {Lattice of combinatorial Hopf algebras: binary trees with
  multiplicities}.
\newblock {\em {Discrete Mathematics \& Theoretical Computer Science}}, {DMTCS
  Proceedings vol. AS, 25th International Conference on Formal Power Series and
  Algebraic Combinatorics (FPSAC 2013)}, January 2013.

\bibitem{Schensted61}
Craige Schensted.
\newblock Longest increasing and decreasing subsequences.
\newblock {\em Canad. J. Math.}, 13:179--191, 1961.

\bibitem{Squier94}
Craig~C. Squier, Friedrich Otto, and Yuji Kobayashi.
\newblock A finiteness condition for rewriting systems.
\newblock {\em Theoret. Comput. Sci.}, 131(2):271--294, 1994.

\bibitem{THOMASYong11}
Hugh Thomas and Alexander Yong.
\newblock Longest increasing subsequences, plancherel-type measure and the
  hecke insertion algorithm.
\newblock {\em Advances in Applied Mathematics}, 46(1):610 -- 642, 2011.

\bibitem{KubatOkninski16}
Łukasz Kubat and Jan Okniński.
\newblock Irreducible representations of the chinese monoid.
\newblock {\em Journal of Algebra}, 466:1 -- 33, 2016.

\end{thebibliography}
\end{small}

\clearpage

\quad

\vfill

\begin{flushright}
\begin{small}
\noindent \textsc{Nohra Hage} \\
\url{nohra.hage@univ-catholille.fr} \\
Faculté de Gestion, Economie \& Sciences (FGES),\\
Université Catholique de Lille,\\
60 bd Vauban, \\
CS 40109, 59016 Lille Cedex, France\\
\bigskip

\noindent \textsc{Philippe Malbos} \\
\url{malbos@math.univ-lyon1.fr} \\
Univ Lyon, Universit\'e Claude Bernard Lyon 1\\
CNRS UMR 5208, Institut Camille Jordan\\
43 blvd. du 11 novembre 1918\\
F-69622 Villeurbanne cedex, France
\end{small}
\end{flushright}

\vspace{0.25cm}

\begin{small}---\;\;\today\;\;-\;\;\hhmm\;\;---\end{small} \hfill
\end{document}